\documentclass[a4paper,english,10pt,reqno,draft]{smfart}
\usepackage[utf8]{inputenc}
\usepackage[T1]{fontenc}
\usepackage{lmodern}
\usepackage{amssymb}
\usepackage{amsmath}
\usepackage{smfthm}
\usepackage{smfenum}
\usepackage{mathrsfs}

\usepackage{graphics}
\usepackage{mdwlist}
\usepackage[all]{xy}

\usepackage{xspace}

\usepackage{comment}
\usepackage[adobe-utopia]{mathdesign}
\usepackage[OT2,OT1]{fontenc}

\def\scr{\mathscr}

\def\Osheaf{\mathscr{O}}

\def\ra{\rightarrow} 
\def\id{\mathrm{id}}

\def\ol{\overline}

\DeclareMathOperator{\Spec}{\mathrm{Spec}}
\def\bbZ{\mathbb{Z}}

\def\sf{\mathsf}
\def\Id{\mathrm{id}}

\newcommand{\catVec}[1]{\mathrm{Vec}_k}

\newcommand{\Homi}{\scr H\kern-2pt\mathpzc{om}}

\newcommand{\bbA}{\mathbb{A}}

\newcommand{\CH}{\mathrm{CH}}
\newcommand{\KMiln}{K^{M}}

\newcommand{\bbG}{\mathbb{G}}
\newcommand{\bbP}{\mathbb{P}}

\newcommand{\Blw}{B}

\newcommand{\sC}{\scr C}
\newcommand{\Cores}{\mathrm{Cores}}
\newcommand{\Res}{\mathrm{Res}}
\newcommand{\bfone}{\mathbf{1}}
\DeclareMathOperator{\colim}{\mathrm{colim}}
\DeclareMathOperator{\Proj}{\mathrm{Proj}}

\newcommand{\catgrC}{\mathbf{grC}}
\newcommand{\redT}{\overline{T^c}}
\newcommand{\sh}{\mathrm{S}}

\newcommand{\Barc}{\mathrm{B}}

\title{Cycle Modules and the Intersection $A_{\infty}$-algebra}
\author{Florian Ivorra} 
\address{Institut de recherche math\'ematique de Rennes\\ UMR 6625 du CNRS\\ Universit\'e de Rennes 1\\
Campus de Beaulieu\\
35042 Rennes cedex (France)}
\email{florian.ivorra@univ-rennes1.fr}

\begin{document}
\frontmatter
\begin{abstract}
In his paper {\itshape Chow Group with Coefficients}, M. Rost has developed a generalization of the classical Chow groups based on Milnor K-theory and an axiomatic generalization of it called cycles modules. For a cycle module $M$ with a ring structure and a smooth scheme $X$ of finite type over a field, we show that Rost's cycle complex with coefficients $C^*(X,M)$ has a structure of an $A_{\infty}$-algebra. In the case of Milnor $K$-theory it provides an homotopy model for the classical intersection theory of cycles.
To construct the $A_{\infty}$-algebra structure we first construct such an $A_{\infty}$-algebra structure on a homotopy invariant version of Rost's complex, this step relies on geometry, and then we apply homological perturbation theory.
\end{abstract}
\begin{altabstract}
Dans son article {\itshape Chow Group with Coefficients}, M. Rost a g\'en\'eralis\'e les groupes de Chow classique \`a partir de la K-th\'eorie de Milnor et une g\'en\'eralisation axiomatique de cette derni\`ere: les modules de cycles. \'Etant donn\'e un module de cycle $M$ muni d'un produit et un sch\'ema lisse $X$ de type fini sur un corps, nous montrons dans ce travail que le complexe de cycle $C^*(X,M)$ \`a coefficients dans $M$ introduit par M. Rost poss\`ede une structure naturelle de $A_{\infty}$-alg\`ebre. Dans le cas particulier de la K-th\'eorie de Milnor cette structure rel\`eve la th\'eorie classique de l'intersection des cycles alg\'ebriques au niveau des complexes. La construction pr\'esente deux facettes, une facette g\'eom\'etrique reposant sur des espaces de d\'eformation au c\^one normal ad hoc et une facette homologique reposant sur le lemme de perturbation.
\end{altabstract}
\thanks{}
\maketitle

\tableofcontents
\mainmatter

\section*{Introduction}
The main basic operation in the intersection theory as developed by W. Fulton \cite{Fulton} is the {\emph{Gysin map}} $f^*:\CH_p(X)\ra\CH_{p-d}(Y)$  associated to a closed regular immersion $f:Y\ra X$ of codimension $d$ between to separated schemes of finite type over a field $k$. This map is obtained as the composition of the specialization map $\CH_p(X)\ra\CH_p(N_YX) $  provided by the {\emph{deformation to the normal cone}} $N_YX$ and the inverse of the pullback map $\CH_{p-d}(Y)\ra\CH_p(N_YX)$. In particular if $X$ is smooth of pure dimension $d$, the diagonal $\Delta_X:X\ra X\times_kX$ being a closed regular immersion of codimension $d$, the associated Gysin map allows to define the intersection of two cycles $\alpha\in\CH_p(X)$ and $\beta\in\CH_q(X)$ as the cycle in $\CH_{p+q-d}(X)$ given by $\alpha\cdot\beta=\Delta^*_X(\alpha\times\beta)$. This intersection product is associative since Gysin maps are functorial. Recall that the usual Chow group $\CH_p(X)$ is the cokernel of the divisor map:
$$\bigoplus_{x\in X_{(p+1)}}\kappa(x)^{\times}\xrightarrow{\mathrm{div}}\bigoplus_{x\in X_{(p)}}\bbZ,$$
which is the last non vanishing differential of the Gerten complex for Milnor K-theory
$$\cdots\ra\bigoplus_{x\in X_{(p+r)}}\KMiln_{r}(\kappa(x)) \ra\cdots\ra \bigoplus_{x\in X_{(p+1)}}\KMiln_1(\kappa(x))\xrightarrow{\mathrm{div}}\bigoplus_{x\in X_{(p)}}\KMiln_0(\kappa(x))\ra 0.$$
In \cite{RostChowCoeff} M. Rost has generalized the classical intersection theory in two directions, first by considering the whole Gersten complex and not only its $0$-homology and secondly by allowing more general coefficients than Milnor K-groups: his so called {\emph{cycle modules}}, which are essentially graded modules over Milnor K-theory endowed with a few extra maps. Given a cycle module $M$, M. Rost builds its intersection theory on the associated cycle complex $C_*(X,M)$, a Gersten like complex with components given by 
$$C_p(X,M,n):=\bigoplus_{x\in X_{(p)}}M_{n+p}(\kappa(x)),$$
entirely in terms of the deformation to the normal cone and four basic maps defined in a pointwise manner at the level of complexes: pullbacks, pushforwards, mutiplication by units and boundary maps. Let $A_p(X,M,n)$ be the $p$-homology of the cycle complex of $X$,  he not only constructs a Gysin map $f^*:A_p(X,M,n)\ra A_{p-d}(Y,M,n+d)$ which coincides with Fulton's map for $n=-p$ and $M=\KMiln_*$, but he also lifts this map to a map of complexes
$$I(f):C_*(X,M,n)\ra C_{*-d}(Y,M,n+d)$$
entirely defined in terms of the four basic maps. To quote \cite[page 324]{RostChowCoeff}:
\begin{quote}
{\itshape
\guillemotleft [...] there is a canonical procedure wich starts from the choice of a coordination of the normal bundle of $f$ and yields a map $I(f)$, as desired, defined in terms of the four basic maps. Different choices lead to homotopic maps $I(f)$, with the homotopy again expressible in terms of the four basic maps\guillemotright.}
\end{quote}
He also proves functoriality by showing that for another closed regular immersion $g:Z\ra Y$ the maps $I(g)\circ I(f)$ and $I(f\circ g)$ are homotopic. Once again he constructs explicitly the homotopy in terms of the four basic maps and a space of {\emph{double deformation to the normal cone}}. By looking closely to the construction of the homotopy between $I(g)\circ I(f)$ and $I(f\circ g)$ given in lemmas 11.6  and 11.7 of \cite{RostChowCoeff}, it is striking that not only does the construction involve solely the four basic maps but that the homotopy itself looks like a Gysin map since it appears to be essentially the composition of some kind of specialization map with a homotopy inverse. In one specializes to the case of the intersection products, this suggests that the lack of associativity of the intersection product given by the Gysin map $I(\Delta_X)$ at the level of the cycle complex is controlled by an higher intersection product.\par
This observation and a loose analogy with symplectic geometry and Floer homology of Lagrangian submanifolds, lead to think that, for a cycle module with a ring structure $M$, the {\emph{cycle complex}} $C^*(X,M)$ of a  smooth scheme $X$ of finite type has a {\emph{natural $A_{\infty}$-algebra structure}}, that is an algebra structure which is only associative up to homotopy and higher homotopies. The aim of this paper is to show that it is indeed the case: our main result is the construction of a family of graded morphism of degree $2-n$ 
$$m^{\cap}_{X,n}:C^*(X,M)^{\otimes n}\ra C^*(X,M)$$
such that for each integer $n\geqslant 1$
$$\sum_{r+s+t=n}(-1)^{r+st}m^{\cap}_{X,r+1+t}\circ\left(\bfone^{\otimes r}\otimes m^{\cap}_{X,s}\otimes\bfone^{\otimes t}\right)=0,$$
where the sum runs over all nonnegative integers $r,s,t$ such that $r+s+t=n$. The map $m^{\cap}_{X,1}$ is the differential of the cycle complex and $m^{\cap}_{X,2}$ is a closed map which induces on cohomology the intersection product defined in \cite{RostChowCoeff}.
To get these {\emph{higher intersection products}} we not only use the deformation to the normal cone and its double variant, but also introduce higher versions of it to take into account all possible diagonal immersions. This {\emph{geometric construction}} done in the first section is designed to deform simultaneously to the normal cone a finite number of closed immersions. If cycle complexes were strongly homotopy invariant, that is if given any vector bundle $E$ of finite rank over $X$ the complex $C^*(E,M)$ of the bundle was isomorphic via pullback to $C^*(X,M)$, this geometric deformation would be enough to construct the $A_{\infty}$-algebra structure solely in terms of the four basic maps. Unfortunately the complex $C^*(E,M)$ is only a strong deformation retract of $C^*(X,M)$ and moreover there is no canonical retraction as soon as $E$ is not trivial: to get a retraction one has to choose a coordination of the bundle $E$, a variant of the usual notion of trivialization needed in \cite{RostChowCoeff} for technical reasons. To overcome this problem the idea is to replace $C^*(X,M) $ by a strong deformation retract that is a strongly homotopy invariant complex and still possesses the formalism of the four basic maps. This is done in section three where we also construct the $A_{\infty}$-algebra structure on this complex. The final step performed in section four consists in a descent along the retraction of the $A_{\infty}$-algebra structure previously obtained on the homotopy invariant cycle complex to the true cycle complex defined by M. Rost, a rather classical problem in {\emph{homological perturbation theory}}. We apply the perturbation lemma of \cite{MR0301736,MR1007895,MR1103672}, a method that has the advantage to give explicit formulas again in terms of the four basic maps.
\subsubsection*{Conventions}
In this work all the schemes are assumed to be separated, of finite type over a field $k$ and from section \ref{ModCycleSection} on they are assumed to be of pure dimension.  We refer to the appendix for our sign conventions and the definition of the tensor DG category $\catgrC$ of $\bbZ$-graded complexes of abelian groups. Coordinates play a role in the construction of \cite{RostChowCoeff} and some formulas of \emph{loc.cit.} do depend on them. So we fix some coordinates $t_1,\ldots,t_n$ on $\bbA^n$, in other words we write $\bbA^n=\Spec(k[t_1,\ldots,t_n])$. Given nonnegative integers $r,s,t$ such that $r+s+t=n$, we denote by $\bbA^r\times_k\{0\}^{s}\times_k\bbA^t$ the closed subscheme $\Spec(k[t_1,\ldots,t_r,t_{r+s+1},\ldots,t_{n}])$ of $\bbA^n$. We identify this scheme with $\bbA^{r+t}$ via the isomorphism of $k$-algebras
$$k[t_1,\ldots,t_r]\ra k[t_1,\ldots,t_r,t_{r+s+1},\ldots,t_n])$$
such that $t_1\mapsto t_1, \ldots, t_r\mapsto t_r$ and $t_{r+1}\mapsto t_{r+s+1},\ldots, t_{r+t}\mapsto t_n$.
\section{Higher deformations to the normal cone} 
Deformation to the normal cone is one of the basic geometrical tool needed to define Gysin maps, the fundamental operation of intersection theory as developed in \cite{Fulton} or \cite{RostChowCoeff}. We refer also to \cite{GilletHandbook} for a nice and short survey of intersection theory. More generally is plays an important role in the definition of various {\emph{specialization maps}} and for example it is a basic ingredient in the microlocal theory of sheaves developed by M. Sato and M. Kashiwara, P. Schapira \cite{MR1074006}: the microlocalization functor is the Fourier transform\footnote{In the sense of Sato and Malgrange-Verdier.} of a specialization functor provided by a space of deformation to the normal cone. The proof of functoriality for Gysin maps in \cite{RostChowCoeff} is achieved  through a space of double deformation to the normal, this is also the approach used by M. Levine in \cite[Part I, Chapter III,\S 2]{MR1623774}. This space allows to deform simultaneously two closed immersions, and similar deformation spaces appears also in the work of K. Takeuchi \cite[\S 2.2]{MR1382804} on bimicrolocalization.\par
In this section we define {\emph{spaces of deformation to the normal cone}} that allow to {\emph{deform simultaneously a finite number of closed immersions}}. Deformations with a similar flavour were considered in \cite{Delort} to study some ramified Cauchy problems within  the microlocal theory of sheaves. In this paper J.-M. Delort develops a notion of simultaneous microlocalisation of a complex of sheaves on a product of real analytic manifolds along a product of real analytic closed submanifolds. 
\subsection{Classical deformation to the normal cone}
 Let us first start by recalling the construction of the deformation to the normal cone $D(X,Y)$ of a closed immersion $\iota:Y\hookrightarrow X$. It is the open complement of the closed subscheme $\Blw_YX$ in the blow-up $D_YX$ of the closed immersion $Y\times_k\{0\}\hookrightarrow X\times_k\bbA^1$. Let $\scr I$ be the ideal of $\Osheaf_X$ corresponding to $Y$. The ideal of definition of $Y\times_k\{0\}$ in $X\times_k\bbA^1$ being
$$(\scr I,t)=\scr I+t\Osheaf_X[t],$$
its $n$-th power is given by
$$(\scr I,t)^n=\scr I^n+\scr I^{n-1}t+\cdots+\scr It+t^n\Osheaf_X[t] $$
and so our blow-up $D_YX$ is the scheme $\Proj((\scr I,t)^n)$. 
The open complement of $\Blw_YX$ in $D_YX$ is then the spectrum of the quasi-coherent $\Osheaf_X[t]$-algebra
$$\scr A:=\bigoplus_{n\in\bbZ}\scr I^nt^{-n}\subset \Osheaf_X[t,t^{-1}]$$
with the convention that $\scr I^n=\Osheaf_X$ for nonpositive $n$. We have $\scr A[t^{-1}]= \Osheaf_X[t,t^{-1}] $ and so a diagram
$$\xymatrix{{} & {X\times_k\bbG_m}\ar[d]\ar[ld]\ar@/^3em/@{=}[dd]\\
{D(X,Y)}\ar[d] & {D(X,Y)|_{\bbG_m}}\ar[l]\ar[d]\ar@{}[ld]|{\square}\\
{X\times_k\bbA^1} & {X\times_k\bbG_m.}\ar[l]} $$
On the other hand\footnote{Here we do not put the sign in the identification as in \cite[\S 10.5]{RostChowCoeff}} $\scr A/t\scr A=\oplus_n\scr I^n/\scr I^{n+1}$ and the fiber over $\{0\}$ of the deformation space is thus the normal cone $C_YX$. We finally have the following diagram
$$\xymatrix{{D(X,Y)|_{\bbG_m}}\ar[r]\ar[d]_{\textrm{iso.}}\ar@{}[rd]|{\square} & {D(X,Y)}\ar[d]\ar@{}[rd]|{\square} & {D(X,Y)|_0}\ar[l]\ar[d] & {C_YX}\ar@{=}[l]\ar[d]\\
{X\times_k\bbG_m}\ar[r] & {X\times_k\bbA^1} & {X}\ar[l] & {Y.}\ar[l]}$$
\subsection{Higher deformation to the normal cone}
The deformation space we introduce in this paragraph is a higher analog of the symmetric double deformation space considered by M. Rost in \S 10.6 of \cite{RostChowCoeff}.
\subsubsection{}
Let $Y$ be a scheme of finite type over $k$. We want to deform to their normal cones simultaneously a family of closed immersions in $Y$:
$$\sigma_1:Y_1\hookrightarrow Y,\ldots,\sigma_n:Y_n\hookrightarrow Y$$
Let $\scr I_i\subset \Osheaf_Y$ be the ideal defining the closed subscheme $Y_i$ in $Y$. Then
$$\scr A:=\bigoplus_{(\alpha_1,\ldots,\alpha_n)\in\bbZ^n}\scr I^{\alpha_1}_1\cdots\scr I^{\alpha_n}_nt_1^{-\alpha_1}\cdots t_n^{-\alpha_n} $$
is a quasi-coherent $\Osheaf_Y[t_1,\ldots,t_n]$-algebra contained in $\Osheaf_Y[t_1,t_1^{-1},\ldots,t_n,t_n^{-1}]$. The simultaneous deformation to the normal cone of the immersions $\sigma_1,\ldots,\sigma_n$ is the affine scheme over $Y\times_k\bbA^n$, corresponding to this quasi-coherent  $\Osheaf_Y[t_1,\ldots,t_n]$-algebra:
$$D(Y,Y_1,\ldots,Y_n):=\Spec\left(\bigoplus_{(\alpha_1,\ldots,\alpha_n)\in\bbZ^n}\scr I^{\alpha_1}_1\cdots\scr I^{\alpha_n}_nt_1^{-\alpha_1}\cdots t_n^{-\alpha_n} \right).$$
Since $\scr A[t_1^{-1},\ldots,t_n^{-1}]=\Osheaf_Y[t_1,t_1^{-1},\ldots,t_n,t_n^{-1}]$, we have the following commutative diagram:
$$\xymatrix{{} & {Y\times{\bbG_m}^{\kern-.8em n}}\ar[d]\ar[ld]\ar@/^6em/@{=}[dd]\\
{D(Y,Y_1,\ldots,Y_n)}\ar[d] & {D(Y,Y_1,\ldots,Y_n)|_{\bbG^n_m}}\ar[l]\ar[d]\ar@{}[ld]|{\square}\\
{Y\times_k\bbA^n} & {Y\times_k{\bbG_m}^{\kern-.8em n}\;.}\ar[l]} $$
\subsubsection{}\label{FuncDefNorm}
The functorial properties of the simultaneous deformation to the normal cone are similar to the functorial properties of the more conventional deformation to the normal cone. Let $Z$ be a $k$-scheme and 
$$\tau_1:Z_1\hookrightarrow Z,\ldots,\tau_n:Z_n\hookrightarrow Z,$$
be a family of closed immersions. Consider a family of $k$-morphisms $f:Z\ra Y$ and $f_i:Z_i\ra Y_i$ such that the squares
\begin{equation}\label{squarefunc}
\begin{tabular}{c}
$$
\xymatrix{{Z_i}\ar[r]^{\tau_i}\ar[d]^{f_i} & {Z}\ar[d]^{f}\\
{Y_i}\ar[r]^{\sigma_i} & {Y}}$$
\end{tabular}
\end{equation}
commute for any $i\in\{1,\ldots,n\}$. Let $\scr K_i$ be the ideal of $\Osheaf_Z$ which defines the closed subscheme $Z$. The natural morphisms $f^*\scr I_i\ra \scr K_i$ of $\Osheaf_Z$-modules induce a morphism of quasi-coherent $\Osheaf_Z[t_1,\ldots,t_n]$-algebras
\begin{equation}\label{MorFuncAlg}
f^*\left(\bigoplus_{(\alpha_1,\ldots,\alpha_n)\in\bbZ^n}\scr I^{\alpha_1}_1\cdots\scr I^{\alpha_n}_nt_1^{-\alpha_1}\cdots t_n^{-\alpha_n}\right)\ra \bigoplus_{(\alpha_1,\ldots,\alpha_n)\in\bbZ^n}\scr K^{\alpha_1}_1\cdots\scr K^{\alpha_n}_nt_1^{-\alpha_1}\cdots t_n^{-\alpha_n},
\end{equation}
and thus a morphism of $Z\times_k\bbA^n$-schemes
$$D'(f,f_1,\ldots,f_n):D(Z,Z_1,\ldots,Z_n)\ra D(Y,Y_1,\ldots,Y_n)|_{Z\times_k\bbA^n}.$$
So the functoriality of the simultaneous deformation space is summed up by the following commutative diagram:
$$\xymatrix@C=1.5cm{{D(Z,Z_1,\ldots,Z_n)}\ar[r]^{D'(f,f_1,\ldots,f_n)}\ar[rd]\ar@/^2em/[rr]^{D(f,f_1,\ldots,f_n)} & {D(Y,Y_1,\ldots,Y_n)|_{Z\times_k\bbA^n}}\ar[r]\ar[d]\ar@{}[rd]|{\square} & {D(Y,Y_1,\ldots,Y_n)}\ar[d]\\
{} & {Z\times_k\bbA^n}\ar[r]^{f\times_k\Id} & {Y\times_k\bbA^n.}}$$
Assume that the squares (\ref{squarefunc}) are cartesian. Then the morphisms $f^*\scr I_i\ra \scr K_i$ are surjective and the morphism (\ref{MorFuncAlg}) is surjective as well. Therefore under the assumption that the squares are cartesian, the morphism $D'(f,f_1,\ldots,f_n) $ is a closed immersion.
\subsubsection{}
For short we will denote by $\scr D$ the deformation space we have just defined, and we will denote also by $t_1,\ldots,t_n\in\Gamma(\scr D,\Osheaf_{\scr D})$ the global functions defined by the chosen coordinates on $\bbA^n$. We let $r,s,t$ be nonnegative integers such that $n=r+s+t$. The closed subscheme of $\scr D$ defined by the equations $t_{r+1}=\cdots=t_{r+s}=0$, is the affine scheme over $\bbA^r\times\{0\}^s\times\bbA^t$:
$$\scr D|_{\bbA^r\times\{0\}^s\times\bbA^t}=\Spec\left(\bigoplus_{(\alpha_1,\ldots,\alpha_n)\in\bbZ^n}\dfrac{\scr I^{\alpha_1}_1\cdots\scr I^{\alpha_n}_n}{\sum_{i=r+1}^{r+s}\scr I^{\alpha_1}_1\cdots\scr I_{i}^{\alpha_i+1}\cdots\scr I^{\alpha_n}_n}t_1^{-\alpha_1}\cdots t_r^{-\alpha_r}t_{r+s+1}^{-\alpha_{r+s+1}}\cdots t_n^{-\alpha_n} \right).$$
To go further and identify this closed subscheme of $\scr D$ itself to some higher deformation space to the normal cone, we need to assume that our closed immersions are transverse enough. A rather strong transversality condition is enough for our purpose and we have chosen to phrase it in terms of local coordinates:
\begin{enumerate}
\item[\bf T:]{locally on $Y$ for the Zariski topology there exists an integer $N\geqslant 1$, an \'etale morphism $$Y\ra\bbA^N=\Spec(k[t_1,\ldots,t_N]),$$ and a partition $1=\ell_0<\cdots<\ell_{n+1}=N$ such that for all $i\in\{1,\ldots,n\}$, the closed subscheme $Y_i$ of $Y$ is defined by the equations $t_{\ell_{i-1}}=\cdots=t_{\ell_i}=0$.}
\end{enumerate}
\begin{rema}\label{RemaT} Under the transversality assumption $\bf{T}$, the closed immersion $\sigma_1,\ldots,\sigma_n$ are regular and moreover for any $(\alpha_1,\ldots,a_n)\in\bbZ^n$ we have
$$\scr I_1^{\alpha_1}\cdots\scr I^{\alpha_n}_n=\scr I_1^{\alpha_1}\cap\cdots\cap\scr I^{\alpha_n}_n.$$
Remark that, if condition $\bf{T}$ holds for the closed immersions $\sigma_1,\ldots,\sigma_n$, it still holds for the closed immersions obtained after a smooth base change on $Y$.
\end{rema}
From now on, assume that condition $\bf{T}$ is satisfied. Let us choose an integer $i\in\{1\;\ldots,n\}$. Then the closed subscheme of $\mathscr D$ defined by $t_i=0$ is  
$$ \scr D|_{\bbA^{i-1}\times\{0\}\times\bbA^{n-i}}=\Spec\left(\bigoplus_{(\alpha_1,\ldots,\alpha_n)\in\bbZ^n}\dfrac{\scr I^{\alpha_1}_1\cdots\scr I^{\alpha_n}_n}{\scr I^{\alpha_1}_1\cdots\scr I_{i}^{\alpha_i+1}\cdots\scr I^{\alpha_n}_n}t_1^{-\alpha_1}\cdots t_{i-1}^{-\alpha_{i-1}}t_{i+1}^{-\alpha_{i+1}}\cdots t_n^{-\alpha_n} \right).$$
By remark \ref{RemaT}, the transversality assumption $\bf{T}$ implies that
\begin{equation*}
\begin{split}
(\scr I^{\alpha_1}_1\cdots\scr I^{\alpha_n}_n)\cap\scr I_i^{\alpha_i+1} & =(\scr I^{\alpha_1}_1\cap\cdots\cap\scr I^{\alpha_n}_n)\cap\scr I_i^{\alpha_i+1}=\scr I^{\alpha_1}_1\cap\cdots\cap\scr I_{i}^{\alpha_i+1}\cap\cdots\cap\scr I^{\alpha_n}_n  \\
&= \scr I^{\alpha_1}_1\cdots\scr I_{i}^{\alpha_i+1}\cdots\scr I^{\alpha_n}_n,
\end{split} 
\end{equation*}
and so we have a natural isomorphism
$$\dfrac{\scr I^{\alpha_1}_1\cdots\scr I^{\alpha_n}_n}{\scr I^{\alpha_1}_1\cdots\scr I_{i}^{\alpha_i+1}\cdots\scr I^{\alpha_n}_n}=\dfrac{\scr I^{\alpha_1}_1\cdots\scr I^{\alpha_n}_n+\scr I_i^{\alpha_i+1}}{\scr I_{i}^{\alpha_i+1}}.$$
\begin{rema}\label{RemaNC}
Let $Z\hookrightarrow Y$ be a closed immersion and $\scr I\subset\Osheaf_Y$ the ideal of definition of $Z$ in $Y$. The normal cone of $Z$ in $Y$ is the affine scheme over $Z$
$$C_ZY=\Spec\left(\bigoplus_{m\in\bbZ}\dfrac{\scr I^m}{\scr I^{m+1}}\right).$$
Let $W\hookrightarrow Y$ be another closed immersion defined by an ideal $\scr K\subset\Osheaf_Y$. Then $C_ZY|_{W\times_YZ}$ is the closed subscheme of $C_ZY$ defined by the ideal of $\Osheaf_{C_ZY}$
$$\widehat{\scr K}=\bigoplus_{m\in\bbZ}\dfrac{\scr K\scr I^m+\scr I^{m+1}}{\scr I^{m+1}}.$$
This construction is compatible with the product of ideals: if $\scr L\subset\Osheaf_Y$ is an ideal, then $\widehat{\scr K\cdot\scr L}=\widehat{\scr K}\cdot\widehat{\scr L} $. In particular the $n$-th power of $\widehat{\scr K}$ is 
$${\widehat{\scr K}}^n=\bigoplus_{m\in\bbZ}\dfrac{\scr K^n\scr I^m+\scr I^{m+1}}{\scr I^{m+1}}.$$
\end{rema}
By remark \ref{RemaNC} applied to the immersion $\sigma_i:Y_i\hookrightarrow Y$, we get an isomorphism of quasi-coherent $\Osheaf_{Y_i}[t_1,\ldots,t_{i-1},t_{i+1},\ldots,t_n]$-algebras between
$$\bigoplus_{(\alpha_1,\ldots,\alpha_n)\in\bbZ^n}\dfrac{\scr I^{\alpha_1}_1\cdots\scr I^{\alpha_n}_n}{\scr I^{\alpha_1}_1\cdots\scr I_{i}^{\alpha_i+1}\cdots\scr I^{\alpha_n}_n}t_1^{-\alpha_1}\cdots t_{i-1}^{-\alpha_{i-1}}t_{i+1}^{-\alpha_{i+1}}\cdots t_n^{-\alpha_n}  $$
and 
$$\bigoplus_{(\alpha_1,\ldots,\alpha_{i-1},\alpha_{i+1}\cdots\alpha_n)\in\bbZ^n}\widehat{\scr I_1}^{\alpha_1}\cdots\widehat{\scr I_{i-1}}^{\alpha_{i-1}}\widehat{\scr I_{i+1}}^{\alpha_{i+1}}\cdots\widehat{\scr I_n}^{\alpha_n}t_1^{-\alpha_1}\cdots t_{i-1}^{-\alpha_{i-1}}t_{i+1}^{-\alpha_{i+1}}\cdots t_n^{-\alpha_n}  $$
We thus get an isomorphism over $Y_{i}\times_k\bbA^{n-1}$ between $\scr D|_{\bbA^{i-1}\times\{0\}\times\bbA^{n-i}}$ and the deformation space
$$D(N(Y,Y_i),N(Y,Y_i)|_{Y_1\times_YY_i},\ldots,N(Y,Y_i)|_{Y_{i-1}\times_YY_i},N(Y,Y_i)|_{Y_i\times_YY_{i+1}},\ldots,N(Y,Y_i)|_{Y_i\times_YY_n}) .$$
Let $Y_{r,s,t}$ be the closed subscheme of $Y$ defined by
$$Y_{r,s,t}:=Y_{r+1}\times_Y\cdots\times_YY_{r+s}.$$
Consider the vector bundle $N_{Y}^{r,s,t}$ over $Y_{r,s,t}$ defined inductively by
$$N_{Y}^{r,s,t}:=\begin{cases} N(Y,Y_{r+1}) & \textrm{if $s=1$}\\
N\left(N_{Y}^{r+1,s-1,t},N_{Y}^{r+1,s-1,t}|_{Y_{r,s,t}}\right)& \textrm{otherwise.}\end{cases} $$
Let $\scr D^{r,s,t}$ be the deformation space
$$\scr D^{r,s,t}:=D\left(N_{Y}^{r,s,t},N_{Y}^{r,s,t}|_{Y_1\times_YY_{r,s,t}},\ldots, N_{Y}^{r,s,t}|_{Y_r\times_YY_{r,s,t}},N_{Y}^{r,s,t}|_{Y_{r,s,t}\times_YY_{r+s+1}},\ldots,N_{Y}^{r,s,t}|_{Y_{r,s,t}\times_YY_n}\right) $$
Since condition  $\bf{T}$ is stable by smooth base change on the target, by induction we get an isomorphism between $\scr D|_{\bbA^r\times\{0\}^s\times\bbA^t} $ and $\scr D^{r,s,t}$ over $Y_{r,s,t}\times_k\bbA^{r+t}$. In the sequel we identify these two deformation spaces via this isomorphism.
\subsection{Deformation of the diagonals}\label{HighDefn}
Now we consider the special case of the diagonal closed immersions. Let $X$ be a smooth scheme of finite type over $k$ and $n\geqslant 2$ be an integer. We consider the following closed immersions
$$\delta^i_{X,n}:=\Id_{X}^{i-1}\times_k\Delta_X\times_k\Id_X^{n-1-i}:X^{n-1}\hookrightarrow X^n\qquad i\in\{1,\ldots,n-1\}$$
where $\Delta_X:X\hookrightarrow X\times_k X$ is the diagonal immersion. We denote by $\scr D_{X,n}$ the corresponding space of simultaneous deformation to the normal cone which is an affine scheme over $X^n\times_k\bbA^{n-1}$. In the sequel we denote by $\omega_{X,n}$ the projection map from $\scr D_{X,n}$ to $X^n\times_k\bbA^{n-1}$. 
Since $X$ is smooth, condition {\bf{T}} holds for this family of closed immersions. By construction the projection map $\omega_{X,n}$ induces an isomorphism
$$\scr D_{X,n}|_{\bbG_m^{n-1}}\simeq X^n\times_k\bbG_m^{n-1}.$$
Throughout the sequel we will identify these two schemes. Moreover the fiber over $\{0\}^{n-1}$ is a vector bundle $\scr N_{X,n}$ over $X$ seen as a closed subscheme of $X^n$ via the diagonal immersion. As before, let $r,s,t$ be nonnegative integers such that $r+s+t=n$, and let $u=r+1+t$. With the notation of the previous subsection, we have
$$(X^n)_{r,s,t}=X^u\hookrightarrow X^n$$
with the immersion given by $\id_{X^r}\times_k\Delta_{X,s}\times_k\Id_{X^t}$, where $\Delta_{X,s}:X\hookrightarrow X^s$ is the diagonal immersion. We also have for $i\in\{1,\ldots,r\}$
$$\xymatrix{{X^{n-1}\times_{\delta^i_{X,n},X^n}(X^n)_{r,s,t}}\ar@{=}[d]\ar[r]  & {(X^n)_{r,s,t}\ar@{=}[d]}\\
{X^{u-1}}\ar[r]^{\delta^i_{X,u}}& {X^u}} $$
and for $i\in\{r+s+1,\ldots,n-1\}$
$$\xymatrix{{(X^n)_{r,s,t}\times_{X^n,\delta^i_{X,n}}X^{n-1}}\ar@{=}[d]\ar[r]  & {(X^n)_{r,s,t}\ar@{=}[d]}\\
{X^{u-1}}\ar[r]^{\delta^i_{X,u}}& {X^u}} $$
We set $\scr N^{r,s,t}_{X,n}:=N^{r,s,t}_{X^n}$. Using the functorial properties of the simultaneous deformation to the normal cone given in \S \ref{FuncDefNorm} we see from the two previous squares and the construction of the deformation space $\scr D^{r,s,t}_{X,n}$, that we a natural morphism $\varpi^{r,s,t}_{X,n}:\scr D^{r,s,t}_{X,n}\ra \scr D_{X,u} $ of schemes over $X^u\times_k\bbA^{u-1}$, which fits into the commutative diagram:
$$\xymatrix{{\scr N_{X,n}}\ar[rr]\ar[dd]\ar@{.>}[rd] & {} & {\scr D^{r,s,t}_{X,n}}\ar[rd]^{\omega^{r,s,t}_{X,n}}\ar[dd]_{\varpi^{r,s,t}_{X,n}} & {}\\
{} & {X^u\times_k\{0\}^{u-1}}\ar@{.>}[rr] & {} & {X^u\times_k\bbA^{u-1}}\\
{\scr N_{X,u}}\ar[rr]\ar@{.>}[ru] & {} & {\scr D_{X,u}}\ar[ru]_{\omega_{X,u}} & {}} $$
in which the squares are cartesian. Moreover via the first vertical arrow $\scr N_{X,n}$ is a vector bundle of finite rank over $\scr N_{X,u}$.
We also have a morphism $\varpi^n_{X,s}$ which fits into the following square:
$$\xymatrix{{\scr D_{X,n}}\ar[r]^{\omega_{X,n}}\ar[d]^{\varpi^n_{X,s}} & {X^n\times_k\bbA^{n-1}}\ar[d]\\
{X^r\times_k\scr D_{X,s}\times_k X^t\times_k\bbA^r\times_k\bbA^t}\ar[r]_{\raisebox{-1.1em}{\footnotesize$\Id_{X^r}\times_k\omega_{X,s}\times_k\Id_{\bbA^r}\times_k\Id_{\bbA^t}$}} & {X^r\times_k X^s\times_k\bbA^{s-1}\times_k X^t\times_k\bbA^r\times_k\bbA^t}} $$
where the second vertical arrow is given by the permutation of factors.
\section{Cycle complexes and the four basic maps}\label{ModCycleSection}
In this section we recall shortly the definition of a cycle module and the four basic operations on the associated cycle complex. For the particular need of this work we have to consider boundary morphisms for generalized boundary triples where the open subset is not necessarily the open complement of the closed subset but may be strictly smaller. Those {\emph{weak boundary morphisms}} share all the properties of usual boundary morphisms except they are not anymore closed as graded maps. 
\subsection{Cycle modules and cycle complexes}
Let $M_*$  be a functor from the category of finitely generated field extensions of $k$ to the category of graded abelian groups. Given a finitely generated extension $F|k$, we will simply denote by $M_*(F)$ the value of  $M_*$ at the extension $F|k$. For finitely generated extensions $F|k$ and $E|F$, we denote by 
$$\Res_{E|F}: M_*(F)\ra M_*(E)$$
the structural morphism, also called the \emph{restriction morphism}, which is by definition a morphism of degree $0$.\par
A ($\bbZ$-graded) \emph{cycle premodule} is a functor $M_*$ from the category of finitely generated field extension of $k$ to the category of graded abelian groups together with the following additional data:
\begin{enumerate}
\item[\bf D2:]{for any finite extension $E|F$ of finitely generated field extensions of $k$ a morphism of degree $0$
$$\Cores_{E|F}:M_*(E)\ra M_*(F)$$
called the \emph{corestriction morphism} or norm map;}
\item[\bf D3:]{for any finitely generated extension $F|k$ a structure  on $M_*(F)$ of left-graded module over the Milnor ring $\KMiln_*(F)$;}
\item[\bf D4:]{for any finitely generated extension $F|k$ and each geometric\footnote{As far as valuations are concerned we will use the conventions and vocabulary of \cite[chapter VI]{MR0120249} as in \cite{RostChowCoeff}; they are not consistent with Bourbaki's conventions \cite{MR0194450}. Recall that for a discrete valuation $v$ of rank $n$ on a finitely generated extension $F|k$, we have the inequality 
$$n+\mathrm{tr.deg}_k(\kappa(v))\leqslant\mathrm{tr.deg}_k(F) .$$
The valuation $v$ is said to be geometric when this inequality is an equality. We refer to \cite{MR2427051} for some useful properties of geometric valuations not proved in \cite{RostChowCoeff}.}  discrete valuation $v$ of rank 1 on $F|k$, a morphism
$$\partial_v:M_*(F)\ra M_*(\kappa(v)) $$
of degree $-1$, called the \emph{residue morphism}.}
\end{enumerate}
All these morphisms also satisfy various compatibility conditions as given in \cite[\S 1]{RostChowCoeff} and we will freely refer to these rules in the sequel. Let $M_*$ be a cycle premodule. Let $X$ be a \emph{normal} irreducible scheme of finite type over $k$ and $x$ be a codimension one point. The function field of $X$ is then the residue field $\kappa(\eta)$ at the generic point $\eta$ of $X$ and the local ring $\Osheaf_{X,x}$ is a discrete valuation ring of rank $1$ with fraction field $\kappa(\eta)$ and residue field $\kappa(x)$, the corresponding valuation is geometric and axiom {\bf{D4}} provides thus a residue morphism of degree $-1$:
$$\partial_x:M_*(\kappa(\eta))\ra M_*(\kappa(x)).$$
Now let $X$ be any $k$-scheme of finite type and $x$ be a point in $X$. The normalisation of the integral subscheme $\ol{\{x\}}$ is an integral normal scheme finite and birational over $\ol{\{x\}}$. Over any codimension one point $y$ of in $\ol{\{x\}}$ there lies finitely many point in the normalisation and they are all of codimension one. For any such point $z$ above $y$, we have a residue morphism
$$\partial_z:M_*(\kappa(x))\ra M_*(\kappa(z))$$
and so a well defined morphism of degree $-1$ from $M_*(\kappa(x))$ to $M_*(\kappa(y))$
\begin{equation}\label{Defdifferential}
\partial^x_y=\sum_{z|y}\Cores_{\kappa(z)|\kappa(y)}\circ\partial_z.
\end{equation}
When $y$ is point in $X$ which is not of codimension one in $\ol{\{x\}}$ we simply set $\partial^x_y=0$.
A cycle module $M_*$ is a cycle premodule which satisfies the following two conditions:
\begin{enumerate}
\item[(\bf FD)]{for any irreductible normal scheme of finite type over $k$ and any element $m\in M_*(\kappa(\eta))$, the set $$\{x\in X^{(1)}: \partial^\eta_x(m)\neq 0\}$$ is finite where $\eta$ is the generic point of $X$;}
\item[(\bf C)]{for any integral local 2-dimensional scheme $X$ we have
$$\sum_{x\in X^{(1)}}\partial^{x}_s\circ\partial^\eta_x=0 $$
where $\eta$ is the generic point and $s$ the closed point of $X$.}
\end{enumerate}
Our main result concerns smooth schemes of finite type over $k$, so to simplify notation, we use from the beginning the codimension indexing for all Rost cycle complexes. We will therefore assume all schemes to be of pure dimension over $k$. Let $M_*$ be a cycle module and $X$ be a scheme of finite type over $k$. We let 
$$C^p(X,M,n)=\bigoplus_{x\in X^{(p)}}M_{n-p}(\kappa(x))$$
and define a morphism
$$d:C^p(X,M,n)\ra C^{p+1}(X,M,n)$$
such that the component along $x\in X^{(p)}$ and $y\in X^{(p+1)}$ is given by the morphism $\partial^{x}_{y}$. The definition has sense thanks to condition {\bf FD} and from {\bf C} it follows that $d$ is a differential. Thus for each integer $n$, we obtain a complex of abelian groups $C^*(X,M,n)$ and we denote by $C^*(X,M)$ the associated graded complex of abelian groups:
$$C^*(X,M):=\bigoplus_{n\in\bbZ}C^*(X,M,n).$$
This complex is $\bbZ$-graded by definition and so belongs to the category $\catgrC$.
\subsection{Operations on cycle complexes}
The \emph{four basic maps} on cycle complexes introduced in \cite[\S 3]{RostChowCoeff} are pushforwards, pullbacks, boundary morphisms, and multiplications by units. The fifth important operation is the product map defined for a cycle module with a ring structure such as for example Milnor K-theory. 
\subsubsection{}\label{FourMaps}
Let $X,Y$ be schemes of finite type over $k$ and $f:Y\ra X$ be a morphism of constant relative dimension $r$. 
\begin{enumerate}
\item{The \emph{pushforward} $f_*:C^*(Y,M)\ra C^*(X,M)$ is the morphism in $\catgrC$ defined by the morphisms of abelian groups 
$$f_*:C^p(Y,M,n)\ra C^{p-r}(X,M,n-r)$$
whose components are given by 
$$(f_*)^y_x=\begin{cases} \Cores_{\kappa(y)|\kappa(x)} & \textrm{if $x=f(y)$ and $\kappa(y)|\kappa(x)$ is finite} \\
0 & \textrm{otherwise.} \end{cases}$$
The pushforward is a graded morphism of degree $(-r,r)$, moreover as proved in \cite[Proposition 4.6 (1)]{RostChowCoeff} it is \emph{closed} when $f$ is \emph{proper}.}
\item{Assume that $f$ is flat. The \emph{pullback} $f^*:C^*(X,M)\ra C^*(Y,M)$ is the morphism in $\catgrC$ defined by the morphisms of abelian groups 
$$f^*:C^p(X,M,n)\ra C^{p}(Y,M,n)$$
whose components are given by 
$$(f^*)^x_y=\begin{cases} \lg\left(\Osheaf_{f^{-1}(Z),y}\right)\cdot\Res_{\kappa(y)|\kappa(x)} & \textrm{if $x=f(y)$} \\
0 & \textrm{otherwise} \end{cases}$$
where $Z=\overline{\{x\}}$ with its reduced scheme structure. The pullback is a \emph{closed} graded morphism of degree $(0,0)$ \cite[Proposition 4.6 (2)]{RostChowCoeff}.}
\item{Let $X$ be a $k$-scheme of finite type and $Z\hookrightarrow X$ be a closed immersion. Denote by $U$ its \emph{open complement}. Assume that $Z$ is of pure codimension $c$ in $X$. We have a  boundary triple $$U\hookrightarrow X\hookleftarrow Z$$
and by taking the $\partial^x_y$ defined in (\ref{Defdifferential}) we get the boundary map defined by M. Rost $\partial^U_Z:C^*(U,M)\ra C^*(Z,M)$ which is a \emph{closed} graded morphism in $\catgrC$ of degree $(1-c,c)$ \cite[Proposition 4.6 (4)]{RostChowCoeff}. }
\item{Let $X$ be scheme of finite type over $k$. We have the evaluation map
$$\Gamma(X,\Osheaf^{\times}_X)\ra\Osheaf_{X,x}^{\times}\ra\kappa(x)^{\times};a\mapsto a(x)$$
and so global units $a_1,\ldots,a_r\in\Gamma(X,\Osheaf_X^{\times})$ provide a symbol $\{a_1(x),\ldots,a_r(x)\}\in\KMiln_r(\kappa(x))$. We have then a graded morphism of bidegree $(0,r)$
$$\{a_1,\ldots,a_r\}:C^*(X,M)\ra C^*(X,M)$$ defined by
$$\{a_1,\ldots,a_r\}^x_y:=\begin{cases} \{a_1(x),\ldots,a_r(x)\}\cdot & \textrm{if $y=x$,}\\
$0$ & \textrm{otherwise}
\end{cases}$$
This morphism is \emph{closed} \cite[Proposition 4.6 (3)]{RostChowCoeff}. }
\end{enumerate}
Pushforwards and pullbacks are functorial. These four basic operations satisfy various formulas proved in \cite[\S 4]{RostChowCoeff}. We recall here without proof the formulas we will use:
\begin{enumerate}
\item{for each cartesian square
$$\xymatrix{{X'}\ar[r]^{g'}\ar[d]^{f'}\ar@{}[rd]|{\square} & {Y'}\ar[d]^{f}\\
{X}\ar[r]^{g} & {Y}} $$
with $f$ proper and $g$ flat both of constant relative dimension, we have $g^*\circ f_*=f'_*\circ g'^* $ \cite[Proposition 4.1 (3)]{RostChowCoeff};}
\item{let $f:Y\ra X$ be a flat morphism of constant relative dimension $r$ and $a\in\Gamma(X,\Osheaf^{\times}_X)$ a unit, we have $f^*\circ\{a\}=\{f^*a\}\circ f^*$ \cite[Lemma 4.3 (2)]{RostChowCoeff};}
\item{given a boundary triple $U\hookrightarrow X\hookleftarrow Z$ and a unit $a\in\Gamma(X,\Osheaf^{\times}_X)$, we have $\partial^U_Z\circ\{j^*a\}=-\{i^*a\}\circ\partial^U_Z$ where $j:U\hookrightarrow X$ is the open immersion and $i:Z\hookrightarrow X$ is the closed immersion \cite[Lemma 4.3 (2)]{RostChowCoeff};}
\item{let $f:Y\ra X$ be a flat morphism of constant relative dimension and  $U\hookrightarrow X\hookleftarrow Z$ be a boundary triple, consider the induced boundary triple and cartesian diagram
$$\xymatrix{{V}\ar@{^(->}[r]\ar@{}[rd]|{\square}\ar[d]^{f'} & {Y}\ar[d]^{f}\ar@{}[rd]|{\square} & {W}\ar@{_(->}[l]\ar[d]^{f''}\\
{U}\ar@{^(->}[r] & {X} & {Z,}\ar@{_(->}[l]} $$
we have $f''^*\circ\partial^U_Z=\partial^V_W\circ f'^*$ \cite[Proposition 4.4 (2)]{RostChowCoeff}.}
\end{enumerate}
\subsubsection{}\label{WeakBoundary}
To construct the intersection $A_{\infty}$ on the cycle complex of a smooth scheme, we need to consider boundary morphisms for slightly generalized boundary triples (called {\emph{weak boundary triples}} in the sequel) where the open subset $U$ is only assumed to be contained in the open complement of the closed subset $Z$. The associated {\emph{boundary morphism}}
$$\partial^U_Z:C^*(U,M)\ra C^*(Z,M)$$
is defined in exactly the same way. It is still a morphism in $\catgrC$ of bidegree $(1-c,c)$ where $c$ is the codimension of $Z$ in $X$ and formulas still hold for weak boundary triples with exactly the same proof. However this graded morphism is {\emph{not closed}} in general as soon as $U$ is strictly contained in the open complement of $Z$.
\subsubsection{}\label{ParDefProduct}
Let $M$ be a cycle module with ring product $\times_{\mu}$, and $X,Y$ be $k$-schemes of finite type. The product is a map in $\catgrC$ of degree $(0,0)$ 
$$\times_{\mu}:C^*(X,M)\otimes C^*(Y,M)\ra C^*(X\times_k Y,M).$$
This graded morphism is given by a collection of maps of abelian groups
$$C^p(X,M,n)\otimes_{\bbZ} C^q(Y,M,m)\ra C^{p+q}(X\times_k Y,M,n+m).$$
These maps are defined as follows. Let $x$ be a point of codimension $p$ in $X$, $y$ be a point of codimension $q$ in $Y$, and take elements $a\in M_{n-p}(\kappa(x))$ and $b\in M_{m-q}(\kappa(y))$. Let $Z=\overline{\{x\}}$ and $W=\overline{\{y\}}$ with their reduced scheme structure. Then $a\times_{\mu}b$ is the element in $C^{p+q}(X\times_k Y,n+m)$ with components
$$(a\times_{\mu}b)_u=\begin{cases} \lg(\Osheaf_{Z\times_k W,u})\cdot\Res_{\kappa(u)|\kappa(x)}(a)\cdot\Res_{\kappa(u)|\kappa(y)}(b) & \textrm{if $u$ is a generic point of $Z\times_k W$;}\\
0 & \textrm{otherwise.}\end{cases}$$
\begin{rema}
It follows from the definition that the product is associative. In the case of Milnor K-theory it is also compatible with the classical product on Chow groups.
\end{rema}
For sake of completeness we include here a proof that the product map is \emph{closed}, this corrects the wrong sign in chain rule 14.4 of \cite{RostChowCoeff}.
\begin{lemm}\label{LemmaChainRule}
The product morphism
$$\times_{\mu}:C^*(X,M)\otimes C^*(Y,M)\ra C^*(X\times_k Y,M)$$
is closed in $\catgrC$.
\end{lemm}
\begin{proof}Let $x$ be a codimension $p$ point in $X$ and $y$ be a codimension $q$ point in $Y$. Let $a$ in $M_{n-p}(\kappa(x))$ and $b$ in $M_{m-q}(\kappa(y))$, we have to prove the equality
\begin{equation}\label{ChainRule}
d(a\times_{\mu}b)=(da)\times_{\mu}b+(-1)^{n+p}a\times_{\mu}(db).
\end{equation}
Let $u'$ be a point of codimension $p+q+1$ in $X\times_k Y$. It amounts to prove that the components along $u'$ of both sides are equal. Denote by $x'$ and $y'$ the projection of $u'$ on $X$ and $Y$. According to the definition of the differential and the product, both of them vanish if we are not in  one of the following two cases:
\begin{enumerate}
\item{$x'\in\overline{\{x\}}^{(1)}$ and $y'=y$;}
\item{$x'=x$ and $y'\in\overline{\{y\}}^{(1)}$.}
\end{enumerate}
In the first case, the $u'$-component of $(da)\times_{\mu}b $ vanishes whereas in the second case the $u'$-component of $a\times_{\mu}(db) $ vanishes. 
Assume that we are in the second case. Then
\begin{equation*}
\begin{split}
d(a\times_{\mu} b) & = \sum_{u,z}\Cores_{\kappa(z)|\kappa(u')}\partial_{v(z)}\left((a\times_{\mu}b)_u\right)\\
&= \sum_{u,z}\lg(\Osheaf_{Z\times_k W,u})\Cores_{\kappa(z)|\kappa(u')}\partial_{v(z)}\left(\Res_{\kappa(u)|\kappa(x)}(a)\cdot \Res_{\kappa(u)|\kappa(y)}(b)\right)
\end{split}
\end{equation*}
where the sum runs over codimension $p+q$ points $u$ in $X\times_kY$ lying above $x,y$ and such that $u'\in\overline{\{u\}}^{(1)}$ and over points $z\in\overline{\{u\}}^N$ lying over $u'$. Fix such points $u$ and $z$ and let $\alpha=\Res_{\kappa(u)|\kappa(x)}(a)$ and $\beta=\Res_{\kappa(u)|\kappa(y)}(b)$.
Rule {\bf{P3}} of \cite{RostChowCoeff} applied to the valuation $v(z)$ and the elements $\alpha\in M_{n-p}(\kappa(u))$ and $\beta\in M_{m-q}(\kappa(u))$ gives
$$\partial_{v(z)}(\alpha\cdot\beta)=\partial_{v(z)}(\alpha)\cdot s^{\pi}_{v(z)}(\beta)+(-1)^{n+p}s^{\pi}_{v(z)}(\alpha)\cdot\partial_{v(z)}(\beta)+\{-1\}\cdot\partial_{v(z)}(\alpha)\cdot\partial_{v(z)}(\beta)$$
where $\pi$ is a uniformizer of the valuation $v(z)$. Since the valuation $v(z)$ on $\kappa(u)$ is trivial on $\kappa(x)$, rule {\bf{R3c}} assures that $\partial_{v(z)}(\alpha)=0$ and thus
$$\partial_{v(z)}(\alpha\cdot\beta)=(-1)^{n+p}s^{\pi}_{v(z)}(\alpha)\cdot\partial_{v(z)}(\beta).$$
Now by rule {\bf{R3d}}, we have $s^{\pi}_{v(z)}(\alpha)=\Res_{\kappa(z)|\kappa(x)}(a)$. Using rule {\bf{R2b}} we get
\begin{equation*}
\begin{split}
\Cores_{\kappa(z)|\kappa(u')}\partial_{v(z)}\left(\alpha\cdot \beta\right) &=(-1)^{n+p}\Cores_{\kappa(z)|\kappa(u')}\left(\Res_{\kappa(z)|\kappa(x)}(a)\cdot\partial_{v(z)}(\beta)\right)\\
&=(-1)^{n+p}\Res_{\kappa(u')|\kappa(x)}(a)\cdot\left(\Cores_{\kappa(z)|\kappa(u')}\partial_{v(z)}(\beta)\right)
\end{split}
\end{equation*}
Therefore we get
\begin{equation*}
\begin{split}
[d(a\times_{\mu}&b)]_{u'} =\sum_{u,z}(-1)^{n+p}\lg(\Osheaf_{Z\times_k W,u})\Res_{\kappa(u')|\kappa(x)}(a)\cdot\left(\Cores_{\kappa(z)|\kappa(u')}\partial_{v(z)}(\beta)\right)\\
&=(-1)^{n+p}\Res_{\kappa(u')|\kappa(x)}(a)\cdot\left(\sum_z\Cores_{\kappa(z)|\kappa(u')}\partial_{v(z)}\left(\sum_u\lg(\Osheaf_{Z\times_k W,u})\Res_{\kappa(u)|\kappa(y)}(b)\right)\right)\\
& =(-1)^{n+p}\Res_{\kappa(u')|\kappa(x)}(a)\cdot(d\circ p_2^*(b))_{u'}
\end{split}
\end{equation*}
where $p_2:Z\times_k Y\ra Y$ is the projection. Now by proposition 4.6 (2) of \cite{RostChowCoeff} the flat pullback $p_2^*$ is a closed map in $\catgrC$, we have $d\circ p_2^*=p_2^*\circ d$, and from this we conclude that
$$[d(a\times_{\mu}b)]_{u'}=(-1)^{n+p}(a\times_{\mu}db)_{u'} $$
A similar computation in the first case yields
$$[d(a\times_{\mu}b)]_{u'}=(da\times_{\mu}b)_{u'} $$
Adding the last two formulas we finaly get equality (\ref{ChainRule}) as desired.
\end{proof}
A  computation similar to the one given in the proof of lemma \ref{LemmaChainRule} yields:
\begin{lemm}\label{LemmaProductBoundary}
Let $U\hookrightarrow X\hookleftarrow Z$ be a weak boundary triple and $Y,Y'$ be $k$-schemes of finite type. We have then a weak boundary triple
$$Y\times_k U\times_k Y'\hookrightarrow Y\times_k X\times_k Y'\hookleftarrow Y\times_k Z\times_k Y' $$
and a commutative square
$$\xymatrix@C=1.5cm{{C^*(Y,M)\otimes C^*(U,M)\otimes C^*(Y',M)}\ar[d]^{\bfone\otimes\partial^U_Z\otimes\bfone}\ar[r]^(.6){\times_{\mu}\circ(\bfone_U\otimes\times_{\mu})} & {C^*(Y\times_k U\times_k Y',M)}\ar[d]^{\partial^{Y\times_k U\times_k Y'}_{Y\times_k Z\times_k Y'}} \\
{C^*(Y,M)\otimes C^*(Z,M)\otimes C^*(Y',M)}\ar[r]^(.6){\times_{\mu}\circ(\bfone_U\otimes\times_{\mu})} & {C^*(Y\times_k Z\times_k Y',M).}} $$
\end{lemm}
\section{Homotopy invariance and higher intersection products}
\subsection{The homotopy invariant cycle complex}
In \cite{RostChowCoeff}, M. Rost shows homotopy invariance for cycle complexes. More precisely he proves that for a vector bundle $E$ of finite rank over a $k$-scheme of finite type $X$, the cycle complex $C^*(E,M)$ is a {\emph{strong deformation retract}} of $C^*(X,M)$. However there is {\emph{no canonical retraction}} as soon as $E$ is not trivial: to get a retraction one has to choose a coordination of the bundle $E$, a variant of the usual notion of trivialization adapted to the use of boundary maps. This dependance disappears at the derived category level, but we have to work at the level of complex and the lack of a canonical homotopy inverse is a {\emph{serious drawback}}. To overcome this problem in this section we replace $C^*(X,M) $ by strong deformation retract that still possesses the formalism of the four basic maps.  Usually homotopy invariance is achieved by considering some kind of Bloch-Suslin complex, instead here we simply consider a colimit.
\subsubsection{}
 We let $\pi_{n,n-1}:\bbA^n\ra\bbA^{n-1}$ be the map of schemes induced by the inclusion of rings $k[t_1,\ldots,t_{n-1}]\hookrightarrow k[t_1,\ldots,t_n]$. For a $k$-scheme $X$ of finite type, we let $\pi_{X,n,n-1}:=\id_X\times_k\pi_{n,n-1}$. We denote by $\pi_{X,n}:X\times_k\bbA^n\ra X$ the projection.
\begin{defi}
Let $X$ be a $k$-scheme of finite type and $M$ be a cycle module over $k$. We let $\sC^*(X,M)$ be the complex of abelian groups
$$\sC^*(X,M)=\colim_n C^*(X\times\bbA^n,M)$$
with structural morphisms in the inductive system given by the flat pullbacks $\pi^*_{X,n,n-1}$.
\end{defi}
By construction we have a closed graded map of degree $(0,0)$:
$$\alpha^{\bf{A}}_X:C^*(X,M)\ra\sC^*(X,M).$$
\subsubsection{}\label{FourMapsHTP} The four basics maps defined by M. Rost in \cite[\S 3]{RostChowCoeff} extend directly to the complex $\sC^*(X,M)$. For sake of completeness we provide some details, but this is straightforward from the various compatibilities between the four basic maps checked in \cite[\S 4]{RostChowCoeff}. Let $X,Y$ be scheme of finite type over $k$ and $f:Y\ra X$ a morphism of constant relative dimension $r$. Flat pull-back being functorial, if $f$ is flat it induces a commutative diagram
$$\xymatrix@R=.6cm{{C^*(X,M)}\ar[r]\ar[d]^{f^*} & {\cdots}\ar[r] & {C^*(X\times_k\bbA^{n-1},M)}\ar[r]^{\pi^*_{X,n,n-1}}\ar[d]^{(f\times_k\Id_{\bbA^{n-1}})^*} & {C^*(X\times_k\bbA^n,M)}\ar[r]\ar[d]^{(f\times_k\Id_{\bbA^n})^*} & {\cdots}\\
{C^*(Y,M)}\ar[r] & {\cdots}\ar[r] & {C^*(Y\times_k\bbA^{n-1},M)}\ar[r]^{\pi^*_{Y,n,n-1}} & {C^*(Y\times_k\bbA^n,M)}\ar[r] & {\cdots}} $$
and taking colimits, we get a closed graded morphism $f^*:\sC^*(X,M)\ra\sC^*(Y,M)$ of bidegree $(0,0)$. If $f$ is proper we get, using the base change formula, a commutative diagram
$$\xymatrix@R=.6cm{{C^*(Y,M)}\ar[r]\ar[d]^{f_*} & {\cdots}\ar[r] & {C^*(Y\times_k\bbA^{n-1},M)}\ar[r]^{\pi^*_{Y,n,n-1}}\ar[d]^{(f\times_k\Id_{\bbA^{n-1}})_*} & {C^*(Y\times_k\bbA^n,M)}\ar[r]\ar[d]^{(f\times_k\Id_{\bbA^n})_*} & {\cdots}\\
{C^{*}(X,M)}\ar[r] & {\cdots}\ar[r] & {C^{*}(X\times_k\bbA^{n-1},M)}\ar[r]^{\pi^*_{X,n,n-1}} & {C^{*}(X\times_k\bbA^n,M)}\ar[r] & {\cdots}} $$
which provides a closed graded morphism $f_*:\sC^*(X,M)\ra\sC^{*}(Y,M)$ of bidegree $(-r,r)$. 
\begin{rema}
From the definition it is obvious that the base change formula holds for the complexes $\sC^*(X,M)$, more precisely for each cartesian square
$$\xymatrix{{X'}\ar[r]^{g'}\ar[d]^{f'}\ar@{}[rd]|{\square} & {Y'}\ar[d]^{f}\\
{X}\ar[r]^{g} & {Y}} $$
with $f$ proper and $g$ flat, we have $g^*\circ f_*=f'_*\circ g'^* $.
\end{rema}
Let $a_1,\ldots,a_s\in\Gamma(X,\Osheaf^{\times}_X)$ be global units on a $k$-scheme $X$ of finite type. According to \cite[Lemma 4.3]{RostChowCoeff} we have a commutative diagram
$$\xymatrix{{C^*(X,M)}\ar[r]\ar[d]^{\{a_1,\ldots,a_s\}} & {\cdots}\ar[r] & {C^*(X\times_k\bbA^{n-1},M)}\ar[r]^{\pi^*_{X,n,n-1}}\ar[d]^{\{\pi_{n-1}^*a_1,\ldots,\pi_{n-1}^*a_s\}} & {C^*(X\times_k\bbA^n,M)}\ar[r]\ar[d]^{\{\pi_n^*a_1,\ldots,\pi_n^*a_s\}} & {\cdots}\\
{C^*(X,M)}\ar[r] & {\cdots}\ar[r] & {C^*(X\times_k\bbA^{n-1},M)}\ar[r]^{\pi^*_{X,n,n-1}} & {C^*(X\times_k\bbA^n,M)}\ar[r] & {\cdots}} $$
and thus a closed graded morphism $\{a_1,\ldots,a_s\}:\sC^*(X,M)\ra\sC^*(X,M)$ of bidegree $(0,-s)$. 
Given a weak boundary triple $U\hookrightarrow W\hookleftarrow Y$, where $Y$ is of pure codimension $c$ in $X$, we have a commutative diagram by \cite[Proposition 4.4]{RostChowCoeff}
$$\xymatrix{{C^*(U,M)}\ar[r]\ar[d]^{\partial^{U}_Y} & {\cdots}\ar[r] & {C^*(U\times_k\bbA^{n-1},M)}\ar[r]^{\pi_{U,n,n-1}^*}\ar[d]^{\partial^{U\times_k\bbA^{n-1}}_{Y\times_k\bbA^{n-1}}} & {C^*(U\times_k\bbA^n,M)}\ar[r]\ar[d]^{\partial^{U\times_k\bbA^n}_{Y\times_k\bbA^n}} & {\cdots}\\
{C^{*}(Y,M)}\ar[r] & {\cdots}\ar[r] & {C^{*}(Y\times\bbA^{n-1},M)}\ar[r]^{\pi^*_{Y,n,n-1}} & {C^{*}(Y\times_k\bbA^n,M)}\ar[r] & {\cdots}} $$
and thus a graded morphism $\partial^U_Y:\sC^*(U,M)\ra\sC^{*}(Y,M)$ of bidegree $(1-c,c)$. If the weak boundary triple is a boundary triple then this morphism is closed. All compatibilities between basic maps proved in \cite[\S 4]{RostChowCoeff} hold true for $\sC^*(X,M)$.
\subsubsection{}\label{ProductHtp}
Assume that $M$ has a ring structure. Via the definitions and rule {\bf{R2a}}, it is easy to check that the products on $C^*(X,M)$
$$\mu_{X,n}:C^*(X,M)^{\otimes n}\ra C^*(X,M)$$
defined in \S \ref{ParDefProduct} by the ring structure on $M$ are compatible with pullbacks, in particular 
$$\mu_{X\times_k\bbA^r,n}\circ(\pi^*_{X,r,r-1})^{\otimes n}=(\pi^*_{X,r,r-1})\circ\mu_{X\times_k\bbA^{r-1},n};$$
and we get a closed graded map of bidegree $(0,0)$
$$\mu^{\bf{A}}_{X,n}:\sC(X,M)^{\otimes n}\ra\sC^*(X^n,M).$$
The homotopy invariant cycle complex has therefore also a product structure.
\subsubsection{}\label{HtpSDR} 
By definition the pullback map $\sC^*(X,M)\ra\sC^*(X\times_k\bbA^1)$ is an isomorphism. More generally we have the following lemma:
\begin{lemm}\label{LemmaHTPBundle}
Let $X$ be a $k$-scheme of finite type and $E$ be a vector bundle of finite rank over $X$. Then the pullback morphism
\begin{equation}\label{PullbackBundle}
\sC^*(X,M)\ra\sC^*(E,M)
\end{equation}
is an isomorphism.
\end{lemm}
\begin{proof}
It is enough to check that the presheaf $Y\mapsto\sC^p(Y,M,m)$ on the small Nisnevich site $X_{\mathrm{Nis}}$ of $X$ is a Nisnevich sheaf. Indeed in that case the presheaf $Y \mapsto\sC^p(E|_Y,M,m)$ on $X_{\mathrm{Nis}}$ is also a Nisnevich sheaf and the proof that (\ref{PullbackBundle}) is an isomorphism can be done locally on $X$ for the Zariski topology, and we are reduced to the case of a trivial bundle of finite rank which follows from the definition of the homotopy invariant cycle complex.
Since $X$ is assumed to be of pure dimension, it is equivalent to show that the presheaf on $X_{\mathrm{Nis}}$
$$Y\mapsto\sC_p(Y,M,m)=\colim_n\sC_{p+n}(Y\times_k\bbA^n,M,m-n)$$
 is a Nisnevich sheaf.  Recall that a presheaf of abelian group $F$ on $X_{\mathrm{Nis}}$ is a Nisnevich sheaf if and only if the square 
 $$\xymatrix{{F(Y)}\ar[r]\ar[d] & {F(V)}\ar[d]\\
{F(U)}\ar[r] & {F(U_V)}} $$
is cartesian for any \'etale morphism $Y\ra X$ and any distinguished cartesian square\footnote{This means that such that $p$ is an \'etale morphism, $e$ is an open immersion and the induced map of closed subschemes with their reduced scheme structure $V\setminus U_V\ra Y\setminus U$ is an isomorphism.}
$$\xymatrix{{U_V}\ar[r]\ar[d]\ar@{}[rd]|{\square} & {V}\ar[d]^p\\
{U}\ar[r]^e & {Y.}} $$
Given an excision square we have decompositions as a direct sums of abelian groups
\begin{equation*}
\begin{split}
C_p(Y,M,n) &=C_p(U,M,m)\oplus C_p(Z,M,m)\\
C_p(V,M,n) &=C_p(U_V,M,m)\oplus C_p(Z_V,M,m)
\end{split}
\end{equation*}
where $Z$ is the closed complement of $U$ in $Y$ with its reduced scheme structure and $Z_V$ is the closed complement of $U_V$ in $V$ with its reduced scheme structure.
Since the squares
$$\xymatrix{{U_V\times_k\bbA^n}\ar[r]\ar[d]\ar@{}[rd]|{\square} & {V\times_k\bbA^n}\ar[d]^p\\
{U\times_k\bbA^n}\ar[r]^e & {Y\times_k\bbA^n}} $$
are also excision squares, we have similar decompositions:
\begin{equation*}
\begin{split}
C_{p+n}(Y\times_k\bbA^n,M,m-n) &=C_{p+n}(U\times_k\bbA^n,M,m-n)\oplus C_{p+n}(Z\times_k\bbA^n,M,m-n)\\
C_{p+n}(V\times_k\bbA^n,M,m-n) &=C_{p+n}(U_V\times_k\bbA^n,M,m-n)\oplus C_{p+n}(Z_V\times_k\bbA^n,M,m-n).
\end{split}
\end{equation*}
Taking the colimit over $n$ we get decompositions as direct sums of abelian groups
\begin{equation*}
\begin{split}
\sC_p(Y,M,m) &=\sC_p(U,M,m)\oplus \sC_p(Z,M,m)\\
\sC_p(V,M,m) &=\sC_p(U_V,M,m)\oplus \sC_p(Z_V,M,m)
\end{split}
\end{equation*}
and the results follows. 
\end{proof}
Let $C$ and $D$ be objects in $\catgrC$. We assume to be given two closed graded maps of bidegree $(0,0)$ 
$$\xymatrix@C=.6cm{{C}\ar@<.5ex>[r]^\alpha & {D}\ar@<.5ex>[l]^r} $$
such that $r\circ\alpha=\bfone$ and a graded map $H:D\ra D$ of bidegree $(-1,0)$ such that $\delta(H)=\bfone-\alpha\circ r$. Such objects and maps are called \emph{SDR-data}. In the terminology of \cite[\S 9.1]{RostChowCoeff} a closed graded map $\alpha:C\ra D$ of bidegree $(0,0)$ is called a strong homotopy equivalence if there is an SDR-datum such that $H\circ\alpha=0$ and the pair $(r,H)$ is called an $h$-data for $\alpha$. To be short a SDR-datum will be denoted by the diagram
$$\left(\xymatrix@C=.6cm{{C}\ar@<.5ex>[r]^\alpha & {D}\ar@<.5ex>[l]^r},H\right).$$
Let us denote by $\alpha_X:C^*(X,M)\ra C^*(X\times_k\bbA^1)$ the pullback map. As shown by M. Rost in \cite[\S 9.1]{RostChowCoeff} the morphism $\alpha_X$ is strong homotopy equivalence and an $h$-data is provided by the two graded maps
$$r_X:C^*(X\times_k\bbA^1,M)\xrightarrow{j^*} C^*(X\times_k\bbG_m,M)\xrightarrow{\{-1/t\}}C^*(X\times_k\bbG_m)\xrightarrow{\partial_{\infty}}C^*(X,M);$$
$$H_X:C^*(X\times_k\bbA^1,M)\xrightarrow{p_2^*}C^*(X\times_k(\bbA^1\times_k\bbA^1\setminus\Delta),M)\xrightarrow{\{s-t\}}\xymatrix@R=.5cm{{C^*(X\times_k (\bbA^1\times_k\bbA^1\setminus\Delta),M)}\ar[d]^{p_{1*}}\\ {C^*(X\times_k\bbA^1,M).}}$$
In the definition above $t$ is the standard coordinate on $\bbA^1$, $s,t$ are the coordinates on $\bbA^1\times\bbA^1$, $\Delta$ is the diagonal, $p_1,p_2$ are the two projections, $j:\bbG_m\hookrightarrow\bbA^1$ is the open embedding and $\partial_{\infty}$ is the boundary morphism with respects to the boundary triple
$$X\times\bbG_m\hookrightarrow X\times_k(\bbP^1\setminus\{0\})\hookleftarrow X\times_k\{\infty\}=X.$$
By induction one defines 
$$r_{X,n}= r_{X,n-1}\circ r_{X\times_k\bbA^{n-1}}\qquad H_{X,n}=H_{X\times_k\bbA^{n-1}}+\pi_{X,n,n-1}^*\circ H_{X,n-1}\circ r_{X\times_k\bbA^{n-1}}.$$
Our definition of a retraction differs slightly from Rost's definition, this choice is more convenient for us. Indeed with our conventions one has
\begin{equation*}
\begin{split}
H_{X,n}\circ\pi^*_{X,n,n-1}&= H_{X\times_k\bbA^{n-1}}\circ\pi^*_{X,n,n-1}+\pi^*_{X,n,n-1}\circ H_{X,n-1}\circ r_{X\times_k\bbA^{n-1}}\circ\pi^*_{X,n,n-1}\\
&= 0+\pi^*_{X,n,n-1}\circ H_{X,n-1}=\pi^*_{X,n,n-1}\circ H_{X,n-1}
\end{split}
\end{equation*}
and therefore we get a graded map of degree $(-1,0)$
$$H_X^{\bf{A}}:=\colim_n H_{X,n}:\sC^*(X,M)\ra \sC^*(X,M).$$
Since $r_{X,n}\circ\pi^*_{X,n,n-1}=r_{X,n-1}$, we also have a closed graded map of degree $(0,0)$
$$r^{\bf{A}}_X:\sC^*(X,M)\ra C^*(X,M).$$
It is easy to check that $\delta(H_{X,n})=\bfone_{X\times_k\bbA^n}-\pi^*_{X,n}\circ r_{X,n} $ and so the graded maps $r^{\bf{A}}_X$ and $H^{\bf{A}}_X$ provide an  $h$-data for $\alpha^{\bf{A}}_X$:
$$\left(\xymatrix@C=.6cm{{C^*(X,M)}\ar@<.5ex>[r]^{\alpha^{\bf{A}}_X} & {\sC^*(X,M)}\ar@<.5ex>[l]^{r^{\bf{A}}_X}},H^{\bf{A}}_X\right).$$ 
\subsection{$\bf{A_{\infty}}$-algebra structure on homotopy invariant cycle complexes}\label{HtpAinf}
In the present subsection we assume that $X$ is a smooth scheme of finite type over $k$. It will be convenient to shorten our notation and write simply $[X]$ to denote the homotopy invariant cycle complex $\sC^*(X,M)$ with coefficients in a cycle module $M$. We assume that $M$ has a \emph{ring structure} and explain how to the higher deformation spaces introduced in section \ref{HighDefn} provide an $A_{\infty}$-algebra structure on $[X]$. 
\subsubsection{}\label{DefnAinf}
First let us recall the definition of an $A_{\infty}$-algebra. The formulation we give below is sufficient for our need in the present section, we will need a more compact but maybe less intuitive reformulation of the definition in the next section.  An \emph{$A_{\infty}$-algebra} is an object\label{DefAinfty} $A$ in $\catgrC$ together with a graded map of bidegree $(2-n,0)$
$$m_n: A^{\otimes n}\ra A$$
for each integer $n\geqslant 1$, such that
\begin{itemize}
\item{$m_1$ coincides with the differential of $A$;}
\item{for any integer $n\geqslant 1$
$$\sum_{r+s+t=n}(-1)^{r+st}m_u\circ\left(\bfone_A^{\otimes r}\otimes m_s\otimes\bfone_A^{\otimes t}\right)=0$$
where the sum runs over all nonnegative integers $r,s,t$ such that $r+s+t=n$ and $u=r+1+t$.}
\end{itemize}
\subsubsection{}
Let $n\geqslant 2$ be an integer. The higher deformation space $\mathscr{D}_{X,n}$ defined in subsection \ref{HighDefn} provides a weak boundary triple
$$X^n\times_k\bbG_m^{n-1}\hookrightarrow\scr D_{X,n}\hookleftarrow\scr N_{X,n} $$
and the fiber $\scr N_{X,n}$ over $\{0\}^{n-1}$ of $\mathscr{D}_{X,n}$ is a vector bundle of finite rank over $X$. Let us denote by $\partial_{X,n}:\left[X^n\times_k\bbG_m^{n-1}\right]\ra[\scr N_{X,n}]$ the boundary morphism associated to this triple and let $\eta_{X,n}:\scr N_{X,n}\ra X$ be the projection. The higher specialization map $J_{X,n}$ is defined as the composition
$$\left[X^n\right]\xrightarrow{(\rho_{X,n})^{*}}\left[X^n\times_k\bbG_m^{n-1}\right]\xrightarrow{\{t_1,\ldots,t_{n-1}\}}\left[X^n\times_k\bbG_m^{n-1}\right]\xrightarrow{\partial_{X,n}}\left[\scr N_{X,n}\right]$$
where $\rho_{X,n}:X^n\times_k\bbG_m^{n-1}\ra X^n$ denotes the projection map. By lemma \ref{LemmaHTPBundle}, we know that the map $\eta_{X,n}^*$ is an isomorphism, we can therefore define our
higher products 
$$m^{\bf{A},\cap}_{X,n}:\sC^*(X,M)^{\otimes n}\ra\sC^*(X,M) $$
to be the maps
$$m^{\bf{A},\cap}_{X,n}=(\eta^*_{X,n})^{-1}\circ J_{X,n}\circ \mu^{\bf{A}}_{X,n}.$$
For $n=1$ the map $m^{\bf{A},\cap}_{X,1} $ is simply the differential of $[X]$. 
\subsubsection{}
Now the main result is:
\begin{theo}\label{HAinfInter}
The maps $m^{\bf{A},\cap}_{X,n}$ satisfy the relation
$$\sum_{r+s+t=n}(-1)^{r+st}m^{\bf{A},\cap}_{X,u}\circ\left(\bfone_X^{\otimes r}\otimes m^{\bf{A},\cap}_{X,s}\otimes\bfone_X^{\otimes t}\right)=0$$
for all integers $n\geqslant 1$. Consequently they define an $A_{\infty}$-algebra structure on $\sC^*(X,M)$. 
\end{theo}
We break the proof in two lemmas. Let 
$$\partial^n_{X,r,s,t}:\left[X^n\times_k\bbG_m^{n-1}\right]\ra\left[\scr D^{r,s,t}_{X,n}|_{\bbG_m^{u-1}}\right]$$
be the boundary map, associated to the weak boundary triple
$$X^n\times_k\bbG_m^{n-1}\simeq\scr D_{X,n}|_{\bbG_m^{n-1}}\hookrightarrow\scr D_{X,n}|_{\bbG_m^r\times_k\bbA^{s-1}\times_k\bbG_m^t}\hookleftarrow\scr D_{X,n}|_{\bbG_m^r\times_k\{0\}^{s-1}\times_k\bbG_m^t}\simeq\scr D^{r,s,t}_{X,n}|_{\bbG_m^{u-1}} $$
and 
$$\partial^{r,s,t}_{X,0}:\left[\scr D^{r,s,t}_{X,n}|_{\bbG_m^{u-1}}\right]\ra\left[\scr N_{X,n}\right]$$
the boundary map associated to the weak boundary triple
$$\scr D^{r,s,t}_{X,n}|_{\bbG_m^{u-1}}\hookrightarrow\scr D^{r,s,t}_{X,n}\hookleftarrow\scr D^{r,s,t}_{X,n}|_{\{0\}^{u-1}}\simeq\scr N_{X,n}.$$
We then get a morphism $\left[X^n\times_k\bbG_m^{n-1}\right]\ra\left[\scr N_{X,n}\right]$ of bidegree $(3-n,0)$:
$$J_{X,n}^{r,s,t}:=\partial^{r,s,t}_{X,0}\circ\partial^n_{X,r,s,t}\circ\{t_1,\ldots,t_{n-1}\}\circ\rho^*_{X,n}.$$
We consider the maps of bidegree $(3-n,0)$:
$$m^{\bf{A},\cap}_{X,r,s,t}=(\eta_{X,n}^*)^{-1}\circ J^{r,s,t}_{X,n}\circ \mu^{\bf{A}}_{X,n}.$$
The following lemma generalizes lemma 11.7 of \cite{RostChowCoeff}:
\begin{lemm}
Let $r,s,t$ be nonnegative integers such that $n=r+s+t$. Then
$$m^{\bf{A},\cap}_{X,r,s,t}=(-1)^{r+st}m^{\bf{A},\cap}_{X,u}\circ\left(\bfone_X^{\otimes r}\otimes m^{\bf{A},\cap}_{X,s}\otimes\bfone_X^{\otimes t}\right)$$
where $u=r+1+t$.
\end{lemm}
\begin{proof}
Denote by $\overline{\partial}_{X,s}$ the boundary morphism provided by the weak boundary triple
$$X^r\times_k X^s\times_k\bbG_m^{s-1}\hookrightarrow X^r\times_k\scr D_{X,s}\times_k X^t\hookleftarrow X^r\times_k\scr N_{X,s}\times_k X^t $$
Using in particular lemma \ref{LemmaProductBoundary}, we have then a commutative diagram:
$$\xymatrix@C=.6cm{{[X]^{\otimes n}}\ar@{{*}{-}{>}}[ddddd]_{\bfone_X^{\otimes r}\otimes m^{\bf{A},\cap}_{X,s}\otimes\bfone_X^{\otimes t}}\ar@{{*}{-}{>}}[r]^{\mu^{\bf{A}}_{X,n}}\ar@{{*}{-}{>}}[rd]_{\kern-2.5em\bfone_X^{\otimes r}\otimes \mu^{\bf{A}}_{X,s}\otimes\bfone_X^{\otimes t}} & {[X^n]}\ar@{{*}{-}{>}}[r]^{\rho^*_{X,n}} & {\left[X^n\times_k\bbG_m^{n-1}\right]}\\
{} & {[X]^{\otimes r}\otimes[X^s]\otimes[X]^{\otimes t}}\ar@{{*}{-}{>}}[d]_{\bfone_X^{\otimes r}\otimes \rho^*_{X,s}\otimes\bfone_X^{\otimes t}} & {\left[X^r\times_k X^s\times_k\bbG_m^{s-1}\times_k X^t\times_k\bbG_m^{u-1}\right]}\ar@{{*}{-}{>}}[u]_{(\varpi^n_{X,s})^*}\\
{} & {[X]^{\otimes r}\otimes\left[X^s\times_k\bbG_m^{s-1}\right]\otimes[X]^{\otimes t}}\ar@{{*}{-}{>}}[d]_{{\bfone_X^{\otimes r}\otimes\{t_{r+1},\ldots,t_{r+s-1}\}\otimes\bfone_X^{\otimes t}}}\ar@{{*}{-}{>}}[r]^{\times_{\mu}\circ(\bfone\otimes\times_{\mu})} & {\left[X^r\times_k X^s\times_k \bbG_m^{s-1}\times_k X^t\right]}\ar@{{*}{-}{>}}[d]^{\{t_{r+1},\ldots,t_{r+s-1}\}}\ar@{{*}{-}{>}}[u]\\
{} & {[X]^{\otimes r}\otimes\left[X^s\times_k\bbG_m^{s-1}\right]\otimes[X]^{\otimes t}}\ar@{{*}{-}{>}}[d]_{\bfone_X^{\otimes r}\otimes\partial_{X,s}\otimes\bfone_X^{\otimes t}}\ar@{{*}{-}{>}}[r]^{\times_{\mu}\circ(\bfone\otimes\times_{\mu})} & {\left[X^r\times_k X^s\times_k \bbG_m^{s-1}\times_k X^t\right]}\ar@{{*}{-}{>}}[d]^{\overline{\partial}_{X,s}}\\
{} & {[X]^{\otimes r}\otimes[\scr N_{X,s}]\otimes[X]^{\otimes t}}\ar@{{*}{-}{>}}[r]^{\times_{\mu}\circ(\bfone\otimes\times_{\mu})} & {\left[X^r\times_k\scr N_{X,s}\times_k X^t\right]}\\
{[X]^{\otimes u}}\ar@{{*}{-}{>}}[ru]^{\kern-2.5em\bfone_X^{\otimes r}\otimes \eta^*_{X,s}\otimes\bfone_X^{\otimes t}}\ar@{{*}{-}{>}}[r]_{\mu^{\bf{A}}_{X,u}}  & {[X^u]}\ar@{{*}{-}{>}}`[ru]_{(\id_{X^r}\times_k\eta_{X,s}\times_k\Id_{X^t})^*}[ru]  & {}}$$ where the nonlabeled arrows are the obvious pullbacks.
From the weak boundary triples
$$\xymatrix{{\scr D^{r,s,t}_{X,n}|_{\bbG_m^{u-1}}}\ar@{}[rd]|{\square}\ar@{^(->}[r]\ar[d]^{\varpi^{r,s,t}_{X,n}} & {\scr D^{r,s,t}_{X,n}}\ar@{}[rd]|{\square}\ar[d]^{\varpi^{r,s,t}_{X,n}} & {\scr N_{X,n}}\ar@{_(->}[l]\ar[d]^{\varpi^{r,s,t}_{X,n}}\\
{X^u\times\bbG_m^{u-1}}\ar@{^(->}[r]  & {\scr D_{X,u}} & {\scr N_{X,u}}\ar@{_(->}[l]} $$
and the extension of \cite[Proposition 4.4]{RostChowCoeff} mentioned in \S \ref{FourMaps}, \ref{WeakBoundary} and \ref{FourMapsHTP}, we get $\partial^{r,s,t}_{X,0}\circ(\varpi^{r,s,t}_{X,n})^*=(\varpi^{r,s,t}_{X,n})^*\circ\partial_{X,u}$. 
Similarly the weak boundary triples
$$\xymatrix@C=.7cm{{X^n\times_k\bbG_m^{n-1}}\ar@{^(->}[r]\ar[d]^{\varpi^n_{X,s}}\ar@{}[rd]|{\square} & {\scr D_{X,n}|_{\bbG_m^r\times_k\bbA^{s-1}\times_k\bbG_m^t}}\ar[d]^{\varpi^n_{X,s}}\ar@{}[rd]|{\square} & {\scr D^{r,s,t}_{X,n}|_{\bbG_m^{u-1}}}\ar@{_(->}[l]\ar[d]^{\varpi^n_{X,s}}\\
{X^r\kern-.2em\times_k\kern-.2em X^s\kern-.2em\times_k\kern-.2em\bbG_m^{s-1}\kern-.2em\times_k\kern-.2em X^t\times_k\kern-.2em\bbG_m^{r+t}}\ar@{^(->}[r]  & {X^r\kern-.2em\times_k\kern-.2em\scr D_{X,s}\kern-.2em\times_k\kern-.2em X^t\times_k\kern-.2em\bbG_m^{r+t}} & {X^r\kern-.2em\times_k\kern-.2em\scr N_{X,s}\kern-.2em\times_k\kern-.2em X^t\kern-.2em\times_k\kern-.2em\bbG_m^{r+t}}\ar@{_(->}[l]} $$
provides the equality $\partial^{n}_{X,r,s,t}\circ(\varpi^n_{X,s})^*=(\varpi^n_{X,s})^*\circ\tilde{\partial}_{X,s} $ where $\tilde{\partial}_{X,s}$ is the boundary morphism associated to the bottom line. From the weak boundary triples
$$\xymatrix@C=.6cm{{X^r\kern-.2em\times_k\kern-.2em X^s\kern-.2em\times_k\kern-.2em\bbG_m^{s-1}\kern-.2em\times_k\kern-.2em X^t\times_k\kern-.2em\bbG_m^{r+t}}\ar@{^(->}[r]\ar[d]\ar@{}[rd]|{\square} & {X^r\kern-.2em\times_k\kern-.2em\scr D_{X,s}\kern-.2em\times_k\kern-.2em X^t\times_k\kern-.2em\bbG_m^{r+t}}\ar[d]\ar@{}[rd]|{\square} & {X^r\kern-.2em\times_k\kern-.2em\scr N_{X,s}\kern-.2em\times_k\kern-.2em X^t\kern-.2em\times_k\kern-.2em\bbG_m^{r+t}}\ar@{_(->}[l]\ar[d]\\
{X^r\times_k X^s\times_k\bbG_m^{s-1}\times_k X^t}\ar@{^(->}[r]  & {X^r\times_k\scr D_{X,s}\times_k X^t} & {X^r\times_k\scr N_{X,s}\times_k X^t}\ar@{_(->}[l]} $$
and the extension of \cite[Proposition 4.4]{RostChowCoeff} mentioned in \S \ref{FourMaps}, \ref{WeakBoundary} and \ref{FourMapsHTP}, we get a commutative square
$$\xymatrix@C=.43cm{{} & {\left[X^n\times_k\bbG_m^{n-1}\right]}\ar@{{*}{-}{>}}[r]^{\{t_{r+1},\ldots,t_{r+s-1}\}} & {\left[X^n\times_k\bbG_m^{n-1}\right]}\\
{\left[X^r\kern-.2em\times_k\kern-.2em X^s\kern-.2em\times_k\kern-.2em \bbG_m^{s-1}\kern-.2em\times_k\kern-.2em X^t\right]}\ar@{{*}{-}{>}}[r]\ar@{{*}{-}{>}}[d]_{\{t_{r+1},\ldots,t_{r+s-1}\}} & {\left[X^r\kern-.2em\times_k\kern-.2em X^s\kern-.2em\times_k\kern-.2em\bbG_m^{s-1}\kern-.2em\times_k\kern-.2em X^t\kern-.2em\times_k\kern-.2em\bbG_m^{u-1}\right]}\ar@{{*}{-}{>}}[d]_{\{t_{r+1},\ldots,t_{r+s-1}\}}\ar@{{*}{-}{>}}[u]^{(\varpi^n_{X,s})^*}  & {}\\
{\left[X^r\kern-.2em\times_k\kern-.2em X^s\kern-.2em\times_k\kern-.2em \bbG_m^{s-1}\kern-.2em\times_k\kern-.2em X^t\right]}\ar@{{*}{-}{>}}[r]\ar@{{*}{-}{>}}[d]_{\overline{\partial}_{X,s}} & {\left[X^r\kern-.2em\times_k\kern-.2em X^s\kern-.2em\times_k\kern-.2em\bbG_m^{s-1}\kern-.2em\times_k\kern-.2em X^t\times_k\bbG_m^{u-1}\right]}\ar@{{*}{-}{>}}[d]_{\tilde{\partial}_{X,s}}\ar@{{*}{-}{>}}`[ru][ruu]^{(\varpi^n_{X,s})^*}  & {}\\
{\left[X^r\kern-.2em\times_k\kern-.2em\scr N_{X,s}\kern-.2em\times_k\kern-.2em X^t\right]}\ar@{{*}{-}{>}}[r] & {\left[X^r\kern-.2em\times_k\kern-.2em\scr N_{X,s}\kern-.2em\times_k\kern-.2em X^t\kern-.2em\times_k\kern-.2em\bbG_m^{u-1}\right]}\ar@{{*}{-}{>}}[r]^{\raisebox{1em}{$\scriptstyle\{t_1,\ldots,t_r,t_{r+s},\ldots,t_{n-1}\}$}} & {\left[X^r\kern-.2em\times_k\kern-.2em\scr N_{X,s}\kern-.2em\times_k\kern-.2em X^t\kern-.2em\times_k\kern-.2em\bbG_m^{u-1}\right]}\\
{[X^u]}\ar@{{*}{-}{>}}[r]_{\rho_{X,u}^*}\ar@{{*}{-}{>}}[u]_{\raisebox{-2em}{$\scriptstyle(\id_{X^r}\times_k\eta_{X,s}\times_k\Id_{X^t})^*$}} & {\left[X^u\times_k\bbG_m^{u-1}\right]}\ar@{{*}{-}{>}}[r]_{\{t_1,\ldots,t_r,t_{r+s},\ldots,t_{n-1}\}}\ar@{{*}{-}{>}}[u]^{\raisebox{2em}{$\scriptstyle(\id_{X^r}\times_k\eta_{X,s}\times_k\Id_{X^t}\times_k\Id_{\bbG_m^{u-1}})^*$}}  & {\left[X^u\times_k\bbG_m^{u-1}\right]}\ar@{{*}{-}{>}}[u]^{(\id_{X^r}\times_k\eta_{X,s}\times_k\Id_{X^t}\times_k\Id_{\bbG_m^{u-1}})^*}}$$
where the nonlabeled arrows are the obvious pullbacks. We also have the following commutative square:
$$\xymatrix@R=.7cm{{} & {\left[X^r\times_k X^s\times_k\bbG_m^{s-1}\times_k X^t\times_k\bbG_m^{u-1}\right]}\ar@{{*}{-}{>}}[d]_{\tilde{\partial}_{X,s}}\ar@{{*}{-}{>}}[r]^(.65){(\varpi^n_{X,s})^*}   & {\left[X^n\times_k\bbG_m^{n-1}\right]}\ar@{{*}{-}{>}}[d]^{\partial^n_{X,r,s,t}} \\
{} & {\left[X^r\times_k\scr N_{X,s}\times_k X^t\times_k\bbG_m^{u-1}\right]}\ar@{{*}{-}{>}}[r]^(.65){(\varpi^n_{X,s})^*} \ar@{{*}{-}{>}}[d]_{\{t_1,\ldots,t_r,t_{r+s},\ldots,t_{n-1}\}}   & {\left[\scr D^{r,s,t}_{X,n}|_{\bbG_m^{u-1}}\right]}\ar@{{*}{-}{>}}[d]^{\{t_1,\ldots,t_r,t_{r+s},\ldots,t_{n-1}\}} \\
{\left[X^u\times_k\bbG_m^{u-1}\right]}\ar@{{*}{-}{>}}[d]_{\partial_{X,u}}\ar@{{*}{-}{>}}[r]_{\raisebox{-1.3em}{$\scriptstyle(\id_{X^r}\times_k\eta_{X,s}\times_k\Id_{X^t}\times_k\Id_{\bbG_m^{u-1}})^*$}}   & {\left[X^r\times_k\scr N_{X,s}\times_k X^t\times_k\bbG_m^{u-1}\right]}\ar@{{*}{-}{>}}[r]^(.65){(\varpi^n_{X,s})^*}   & {\left[\scr D^{r,s,t}_{X,n}|_{\bbG_m^{u-1}}\right]}\ar@{{*}{-}{>}}[d]^{\partial^{r,s,t}_{X,0}} \\
{\left[\scr N_{X,u}\right]}\ar@{{*}{-}{>}}[rr]_{(\varpi^{r,s,t}_{X,n})^*} & {} & {\left[\scr N_{X,n}\right]}\\
{[X]}\ar@{{*}{-}{>}}`[rru]_{\eta_{X,n}^*}[rru] \ar@{{*}{-}{>}}[u]^{\eta_{X,u}^*}   & {} & {}}$$
Since Milnor K-theory is graded-commutative
\begin{equation*}
\begin{split}
\{t_1,\ldots,t_{n-1}\}&=(-1)^{(s-1)t}\{t_1,\ldots,t_r,t_{r+s},\ldots,t_{n-1},t_{r+1},\ldots,t_{r+s-1}\}\\
&=(-1)^{(s-1)t}\{t_1,\ldots,t_r,t_{r+s},\ldots,t_{n-1}\}\circ\{t_{r+1},\ldots,t_{r+s-1}\}.
\end{split}
\end{equation*}
Now $\partial^n_{X,r,s,t}$ being a boundary morphism we have 
$$\partial^n_{X,r,s,t}\circ\{t_1,\ldots,t_r,t_{r+s},\ldots,t_{n-1}\}=(-1)^{r+t}\{t_1,\ldots,t_r,t_{r+s},\ldots,t_{n-1}\}\circ\partial^n_{X,r,s,t};$$
and therefore
\begin{equation*}
\begin{split}
\partial^n_{X,r,s,t}\circ \{t_1,\ldots,t_{n-1}\}&=(-1)^{(s-1)t}\partial^n_{X,r,s,t}\circ\{t_1,\ldots,t_r,t_{r+s},\ldots,t_{n-1}\}\circ\{t_{r+1},\ldots,t_{r+s-1}\}\\
&=(-1)^{r+st}\{t_1,\ldots,t_r,t_{r+s},\ldots,t_{n-1}\}\circ\partial^n_{X,r,s,t}\circ\{t_{r+1},\ldots,t_{r+s-1}\}.
\end{split}
\end{equation*}
This together with the previous diagrams proves that
$$m^{\bf{A},\cap}_{X,r,s,t}=(-1)^{r+st}m^{\bf{A},\cap}_{X,u}\circ\left(\bfone^{\otimes r}\otimes m^{\bf{A},\cap}_{X,s}\otimes\bfone^{\otimes t}\right)$$
and the lemma is shown.
\end{proof}
Now the following lemma has to be viewed as a higher order generalization of lemma 11.6 of \cite{RostChowCoeff}:
\begin{lemm}
We have 
$$\sum_{r+s+t=n}m^{\bf{A},\cap}_{X,r,s,t}=0$$
where $u=r+1+t$ and the sum is taken over all nonnegative integers $r,s,t$ such that $r+s+t=n$.
\end{lemm}
\begin{proof}
For each decomposition $n=r+s+t$ where $r,s,t$ are nonnegative integers, the subscheme $\bbG_m^r\times_k\{0\}^{s-1}\times_k\bbG_m^t $ of $\bbA^{n-1}$  is a  locally closed subset $B_{r,s,t}$. The various $B_{r,s,t}$ are disjoint locally closed subsets in $\bbA^{n-1}$ and we consider their union $B$ with its reduced scheme structure. We have a decomposition into a direct sum of abelian groups:
$$C_p\left(\scr D_{X,n}|_B,M,m\right)=\bigoplus_{r+s+t=n}C_p\left(\scr D_{X,n}|_{B_{r,s,t}},M,m\right).$$
The fact that $C_*(\scr D_{X,n}|_B,M,m) $ is a complex, implies that its differential $d$ satisfies $d^2=0$. We therefore have
$$\sum_{r+s+t=n}\partial^{r,s,t}_{X,0}\circ\partial^n_{X,r,s,t}=0.$$
Precomposing with $\{t_1,\ldots,t_{n-1}\}\circ\rho^*_{X,n}$ we get 
$$\sum_{r+s+t=n}J_{X,r,s,t}=0.$$
Now the result follows from the definition of the maps $m^{\bf{A},\cap}_{X,r,s,t}$.
\end{proof}
\subsection{Projection formula for higher intersection products}
As one may expect, higher intersection products satisfy some kind of projection formula. More precisely the following proposition holds:
\begin{prop}\label{ProjFormulaHomotopy}
Let $Y\ra X$ be a proper and flat morphism of smooth $k$-schemes of finite type. Let $u\geqslant 1$ be an integer and $r,t$ be nonnegative integers such that $r+1+t=u$. We have
$$m^{\bf{A},\cap}_{X,u}\circ\left(\bfone_X^{\otimes r}\otimes f_*\otimes\bfone_X^{\otimes t}\right)=f_*\circ m^{\bf{A},\cap}_{Y,u}\circ\left((f^*)^{\otimes r}\otimes\bfone_Y\otimes(f^*)^{\otimes t}\right).$$
\end{prop}
\begin{proof}
We may assume $u\geqslant 2$. Otherwise the formula states simply that $f_*$ is a closed morphism. Moreover since products are commutative, it is enough to prove the following formula: 
$$m^{\bf{A},\cap}_{X,u}\circ\left(f_*\otimes\bfone_X^{\otimes u-1}\right)=f_*\circ m^{\bf{A},\cap}_{Y,u}\circ\left(\bfone_Y\otimes(f^*)^{\otimes u-1}\right).$$
Let $\Gamma_f:Y\hookrightarrow Y\times_k X$ be the graph of the morphism $f$ and consider the closed immersions $\theta_i$ given by the cartesian squares
$$\xymatrix{{Y\times_k X^{u-2}}\ar[r]^{\theta_i}\ar[d]_{h_i=f\times_k\Id}\ar@{}[rd]|{\square} & {Y\times_k X^{u-1}}\ar[d]^{h=f\times_k\Id}\\ 
{X^{u-1}}\ar[r]^{\delta^i_{X,u}} & {X^u.}} $$
Let $\scr D$ be the simultaneous deformation space to the normal cone associated to the closed immersions $\theta_1,\ldots,\theta_u$ and $\scr N$ its fiber over $\{0\}^{u}$.
By functoriality as recalled in \S \ref{FuncDefNorm}, we have a commutative diagram:
$$\xymatrix@C=1.5cm{{\scr D}\ar[r]_(.3){D'(h,h_1,\ldots,h_u)}\ar`d[rd][rd]\ar@/^2em/[rr]^{D(h,h_1,\ldots,h_u)} & {\scr D_{X,u}|_{Y\times_k X^{u-1}\times_k\bbA^{u-1}}}\ar[r]\ar[d]\ar@{}[rd]|{\square} & \scr D_{X,u}\ar[d]\\
{} & {Y\times_k X^{u-1}\times_k\bbA^{u-1}}\ar[r]^{f\times_k\Id} & {X^{u}\times_k\bbA^{u-1}.}} $$
The morphism $D'(h,h_1,\ldots,h_u) $ is a closed immersion and therefore $D(h,h_1,\ldots,h_u)$ is a proper morphism.
The morphism $\scr N\ra \scr N_{X,u}|_Y$ induced by $D'(h,h_1,\ldots,h_u)$
being a closed immersion between two vector bundles of the same rank, is an isomorphism. On the other hand for any $i\in\{1,\ldots,u\}$ we have  a square
$$\xymatrix{{Y^{u-1}}\ar[r]^{\delta_{Y,u}^i}\ar[d]_{g_i=\Id\times_k f^{u-2}} & {Y^u}\ar[d]^{g=\Id\times_k f^{u-1}}\\ 
{Y\times_k X^{u-2}}\ar[r]^{\theta_i} & {Y\times_k X^{u-1}}}$$
and so by functoriality a diagram
$$\xymatrix@C=1.5cm{{\scr D_{Y,u}}\ar[r]_(.4){D'(g,g_1,\ldots,g_u)}\ar`d[rd][rd]\ar@/^2em/[rr]^{D(g,g_1,\ldots,g_u)} & {\scr D|_{Y^u\times_k\bbA^{u-1}}}\ar[r]\ar[d]\ar@{}[rd]|{\square} & {\scr D}\ar[d]\\
{} & {Y^{u}\times_k\bbA^{u-1}}\ar[r]^(.4){\Id\times_k f^{u-1}\times_k\Id} & {Y\times_k X^{u-1}\times_k\bbA^{u-1}.}} $$
in which all maps are flat. Now let $\partial:\left[Y\times_k X^{u-1}\times_k\bbG_m^{u-1}\right]\ra[\scr N]$ be the boundary map provided by the weak boundary triple
$$Y\times_k X^{u-1}\times_k\bbG_m^{u-1}\hookrightarrow \scr D\hookleftarrow \scr N.$$
It is easy to check that we have two commutative squares
$$\xymatrix{{[Y]\otimes [X]\otimes\cdots\otimes[X]}\ar[r]\ar[d]^{f_*\otimes\bfone_X^{\otimes u-1}}  & 
{\left[Y\times_k X^{u-1}\right]}\ar[d]^{h_*}\\
{[X]\otimes[X]\otimes\cdots\otimes [X]}\ar[r]& {[X^u]}}\qquad
\xymatrix{{[Y]\otimes [X]\otimes\cdots\otimes[X]}\ar[r]\ar[d]^{\bfone_Y\otimes(f^*)^{\otimes u-1}}  & 
{\left[Y\times_k X^{u-1}\right]}\ar[d]^{g^*}\\
{[Y]\otimes[Y]\otimes\cdots\otimes [Y]}\ar[r]& {[Y^u]}} $$
where the horizontal arrows are given by the product maps on homotopy invariant cycle complexes.
Let $\eta:\scr N\ra Y$ and $\rho:Y\times_k X^{u-1}\times_k\bbG_m^{u-1}\ra Y\times_k X^{u-1}$ be the respective projections, we have then a commutative diagram:
$$\xymatrix@C=.6cm@R=1.7cm{{\left[Y^u\right]}\ar[r]^(.4){\rho_{Y,u}^*}\ar@{}[rd]|{\textrm{\fbox{Sq1}}} & {\left[Y^u\times_k\bbG_m^{u-1}\right]}\ar[r]^{\{t_1,\ldots,t_{u-1}\}}\ar@{}[rd]|{\textrm{\fbox{Sq2}}} & {\left[Y^u\times_k\bbG_m^{u-1}\right]}\ar[r]^(.6){\partial_{Y,u}}\ar@{}[rd]|{\textrm{\fbox{Sq3}}} & {[\scr N_{Y,u}]}\ar@{}[rd]|{\textrm{\fbox{Sq4}}} & {}\\
{\left[Y\times_k X^{u-1}\right]}\ar[r]^(.4){\rho^*}\ar[u]^{g^*}\ar[d]^{h_*}\ar@{}[rd]|{\textrm{\fbox{Sq5}}} & {\left[Y\times_k X^{u-1}\times_k\bbG_m^{u-1}\right]}\ar[r]^{\raisebox{1em}{$\scriptstyle\{t_1,\ldots,t_{u-1}\}$}}\ar[u]^{(g\times_k\Id)^*}\ar[d]^{(h\times_k\Id)_*}\ar@{}[rd]|{\textrm{\fbox{Sq6}}} & {\left[Y\times_k X^{u-1}\times_k\bbG_m^{u-1}\right]} \ar[r]^(.7){\partial}\ar[u]^{(g\times_k\Id)^*}\ar[d]_{(h\times_k\Id)_*}\ar@{}[rd]|{\textrm{\fbox{Sq7}}} & {[\scr N]}\ar[u]^(.26){D(g,g_1,\ldots,g_u)^*}\ar[d]_(.26){D(h,h_1,\ldots,h_u)_*}\ar@{}[rd]|{\textrm{\fbox{Sq8}}} & {[Y]}\ar[d]^{f_*}\ar[l]_{\eta^*}\ar`u[lu][lu]_(.3){\eta_{Y,u}^*} \\
{[X^u]}\ar[r]_(.4){\rho_{X,u}^*} & {\left[X^u\times_k\bbG_m^{u-1}\right]}\ar[r]_{\{t_1,\ldots,t_{u-1}\}} & {\left[X^u\times_k\bbG_m^{u-1}\right]}\ar[r]_(.6){\partial_{X,u}} & {[\scr N_{X,u}]} & {[X]}\ar[l]^{\eta_{X,u}^*}} $$
The commutativity of squares Sq1 and Sq4 follows from functoriality of pullbacks. The commutativity of squares Sq5 and Sq8 follows from the base change formula proved in proposition 4.1.(3) of \cite{RostChowCoeff}, while the commutativity of squares Sq2 and Sq6 is a consequence of lemmas 4.2.(1) and 4.3.(1) of \emph{loc.cit.}. The commutativity of squares Sq3 and Sq5 is provided by proposition 4.4 of \emph{loc.cit.}, going back to the definitions it proves the desired formula. 
\end{proof}
\section{Perturbation theory and intersection theory for cycle complexes}
In the previous section we have explained how to construct an $A_{\infty}$-algebra structure on the homotopy invariant cycle complex of a smooth $k$-scheme of finite type over $k$ with coefficients in a cycle module $M$ with a ring structure. Now we would like to deduce from it an $A_{\infty}$-algebra structure on Rost's cycle complex via the SDR-datum of \S \ref{HtpSDR}:
$$\left(\xymatrix@C=.6cm{{C^*(X,M)}\ar@<.5ex>[r]^{\alpha^{\bf{A}}_X} & {\sC^*(X,M)}\ar@<.5ex>[l]^{r^{\bf{A}}_X}},H^{\bf{A}}_X\right).$$ 
So we want to
\begin{itemize}
\item{descend the $A_{\infty}$-algebra structure from the big complex $\sC^*(X,M)$ to the smaller one $C^*(X,M)$, in order to obtain the desired higher intersection products
$$m^{\cap}_{X,n}:C^*(X,M)^{\otimes n}\ra C^*(X,M)$$
on Rost's cycle complex;}
\item{construct two morphisms of $A_{\infty}$ algebras 
$$\xymatrix@C=.6cm{{C^*(X,M)}\ar@<.5ex>[r]^{\alpha_X} & {\sC^*(X,M)}\ar@<.5ex>[l]^{r_X}} $$
such that:
\begin{enumerate}
\item{$r^{\infty}_X\circ\alpha^{\infty}_X=\bfone$;}
\item{$\alpha^{\infty}_{X,1}=\alpha_X$ and $r^{\infty}_{X,1}=r_X$;}
\item{the morphisms of $A_{\infty}$-algebras $\bfone$ and $\alpha_X^{\infty}\circ r^{\infty}_X$ are homotopic.}
\end{enumerate}}
\end{itemize}
This is a classical problem in \emph{homological perturbation theory} which goes back to J. Stasheff and the early development of $A_{\infty}$-algebras. Our main references will be \cite{MR0301736,MR893160,MR1103672}, papers which contain the results needed for the present application. There are other ways to solve this question, see \cite{MR1854636} and \cite{MR1868174} for instance, however perturbation lemma and its application nicknamed \emph{tensor trick} have a key advantage for us: they provide explicit formulas for which we can checked that the new higher intersection products still satisfy basic properties such as the projection formula.\par
For readers' convenience we have included in the next two subsections a survey of the application of the perturbation lemma to our present situation. The explicit formulas that are obtained in terms of the \emph{bar construction} are used in the third subsection to prove the projection formula. 
\subsection{The bar construction}
Let us now recall the definition of an $A_{\infty}$-category in term of the bar construction.
\subsubsection{} A coalgebra in $\catgrC$ is the data of an object $A$ and a comultiplication $\Delta:A\ra A\otimes A$ of bidegree $(0,0)$ which is coassociative \emph{i.e.} $(\Delta\otimes\bfone)\circ\Delta=(\bfone\otimes\Delta)\circ\Delta$. A coderivation of a coalgebra is a graded map $b:A\ra A$ which satisfies the analog of the Leibniz rule: $\Delta\circ b=(\bfone\otimes b+b\otimes\bfone)\circ\Delta$. Let $C$ be an object in $\catgrC$. The reduced tensor coalgebra $\redT(C)$ of $C$ is defined as the object of $\catgrC$:
$$\redT(C)=\bigoplus_{n\geqslant 1}C^{\otimes n}$$
together with the comultiplication $\Delta:\redT(C)\ra\redT(C)\otimes\redT(C)$ whose $n$-th component $C^{\otimes n}\ra\redT(C)\ra\redT(C)\otimes\redT(C)$ is  the sum of the morphisms $C^{\otimes n}\ra C^{\otimes i}\otimes C^{\otimes j}$ with $i+j=n$ provided by the associativity constraint of $\catgrC$.
\begin{rema}\label{RemCoDer}
Each graded map  $b:\redT(C)\ra C$ lifts uniquely to a coderivation $\redT(C)\ra\redT(C)$ of the same bidegree. More precisely its lifting to a coderivation of $\redT(C)$ is obtained as follows: each decomposition $n=r+s+t$ of an integer $n\geqslant 1$ into a sum of nonnegative integers provides a graded map $\bfone^{\otimes r}\otimes b_s\otimes \bfone^{\otimes t}:C^{\otimes n}\ra C^{\otimes u}$ where $b_s:C^{\otimes s}\ra C$ is the $s$-th component of $b$, the sum of these graded maps provides a graded morphism
$$\sum_{r+s+t=n}\bfone^{\otimes r}\otimes b_s\otimes \bfone^{\otimes t}:C^{\otimes n}\ra\redT(C)$$
and the collection of these graded morphisms as $n$ varies defines the coderivation 
$$\sum_{n\geqslant 1}\sum_{r+s+t=n}\bfone^{\otimes r}\otimes b_s\otimes \bfone^{\otimes t}:\redT(C)\ra\redT(C)$$
which lifts the morphism $b$ we have started with.
\end{rema}
\subsubsection{}\label{Desuspension}
To get rid of the signs in the definition of an $A_{\infty}$-algebra given in \S \ref{DefnAinf}, it is useful to introduce the suspension functor. 
Let $n\geqslant 1$ be an integer and $C,D$ two objects in $\catgrC$. We have an isomorphism
$$\catgrC^{(r,s)}\left(C^{\otimes i},D\right)\xrightarrow{\sim}\catgrC^{(r+i-1,s)}\left((\sh C)^{\otimes i},\sh D\right) $$
which maps a graded morphism $f:C^{\otimes i}\ra D$ to the graded morphism
\begin{equation}\label{BijSign}
(-1)^{r+s+i-1}s_D\circ f\circ(s^{-1}_C)^{\otimes i}:(\sh C)^{\otimes i}\ra\sh D;
\end{equation}
where $s_C$ and $s_D$ are defined as in the appendix.
\subsubsection{} The bar coalgebra of an object $A$ in $\catgrC$ is the reduced tensor coalgebra $\Barc(A):=\redT(\sh A)$. Assume to be given a family of graded maps $m_n:A^{\otimes n}\ra A$ of bidegree $(2-n,0)$. The above isomorphism provides a family of graded maps $b_n:(\sh A)^{\otimes}\ra \sh A$ of bidegree $(1,0)$ and thus a graded map $\Barc(A)\ra \sh A$ of bidegree $(1,0)$ 
which is the restriction of a unique coderivation $b:\Barc(A)\ra \Barc(A)$ of degree $(1,0)$ according to remark \ref{RemCoDer}.\par
Using the formula above, one sees that the maps $m_n$ define a $A_{\infty}$-algebra structure on $C$ if and only if the corresponding coderivation $b$ on $\Barc(C)$ is a differential \emph{i.e.} satisfies $b^2=0$:
\begin{lemm}
With the notation above the following are equivalent:
\begin{enumerate}
\item{the maps $m_n:A^{\otimes n}\ra A$ yield an $A_{\infty}$-structure on $A$;}
\item{the corresponding coderivation $b$ on $\Barc(A)$ satisfies $b^2=0$;}
\item{for each $n\geqslant 1$, we have \footnote{Note that no signs appear in this formula. The signs in the definition of an $A_{\infty}$ come from the bijection \ref{BijSign} and the sign which occurs in particular in the definition of the suspension $\sh$.}
$$\sum_{r+s+t}b_u \circ(\bfone^{\otimes r}\otimes b_s\otimes\bfone^{\otimes t}),$$
where the sum runs over all decompositions $n=r+s+t$ and we set $u=r+1+t$.}
\end{enumerate}
\end{lemm}
Let $A$ and $B$ be two $A_{\infty}$-algebras. Via isomorphism (\ref{BijSign}), it is not difficult to see that $A_{\infty}$-morphisms $A\ra B$ are in one to one correspondence with maps of DG coalgebras $\Barc(A)\ra\Barc(B)$.
\subsection{Perturbation lemma and tensor trick}\label{Perturbation}
Let $C$ and $D$ be objects in $\catgrC$. Let 
\begin{equation}\label{SDRdatum}
\left(\xymatrix@C=.6cm{{C}\ar@<.5ex>[r]^\alpha & {D}\ar@<.5ex>[l]^r},H\right)
\end{equation}
be an SDR-datum.
\subsubsection{}\label{Contra} This SDR-datum is said to be a \emph{contraction} (or to satisfy the \emph{side conditions}) if its homotopy has the following properties:
\begin{enumerate}
\item{$H\circ\alpha=0$;}
\item{$r\circ H=0$;}
\item{$H^2=0$.}
\end{enumerate}
As noticed in \cite[\S 2.1]{MR893160}, it is always possible to alter the homotopy of a given a SDR-datum, in order to get a contraction. Namely first we consider the homotopy $H'=\delta(H)\circ H\circ\delta(H)$ which satisfies conditions (1) and (2), and then we get a contraction
$$\left(\xymatrix@C=.6cm{{C}\ar@<.5ex>[r]^\alpha & {D}\ar@<.5ex>[l]^r},H''\right) $$
by taking $H''=H'\circ d\circ H'$.
\subsubsection{}\label{Trickformula}
We assume here that the SDR-datum (\ref{SDRdatum}) is a contraction and that $D$ has an $A_{\infty}$-algebra structure. Consider the induced contraction between the corresponding bar DG coalgebras \footnote{The differential $d_C$ on the bar coalgebra $\Barc(C)$ (likewise for $D$) is gotten from the family of maps $d_{C,n}=0$ for $n\geqslant 2$ and $d_{C,1}=d_C$.}
\begin{equation}\label{CoalContr}
\left(\xymatrix@C=.6cm{{(\Barc(C),d_C)}\ar@<.5ex>[r]^{\Barc(\alpha)} & {(\Barc(D),d_D)}\ar@<.5ex>[l]^{\Barc(r)}},\Barc(H)\right);
\end{equation}
where $\Barc(\alpha),\Barc(r)$ are the obvious maps and $\Barc(H)$ denotes abusively the maps given on each factor $(\sh D)^{\otimes n}$ by
$$\sum_{i=1}^n\bfone^{\otimes i-1}\otimes \sh(H)\otimes (\sh(\alpha)\circ\sh(r))^{\otimes n-i}:(\sh D)^{\otimes n}\ra (\sh D)^{\otimes n}.$$
The maps $\Barc(\alpha),\Barc(r) $ are maps of coalgebras and $\Barc(H)$ is an homotopy of DG coalgebras. The structure of $A_{\infty}$-algebra on $D$ provides another differential $b_D$ on $D$. Let $t_D$ be the difference $t_D=b_D-d_D$. Following \cite[\S 3]{MR0301736}, one then defines inductively a sequence of graded maps $t_{D}^{(p)}$ 
$$t^{(1)}_{D}=t_D,\quad t_{D}^{(p+1)}=(t_D\circ \Barc(H))^{\circ p}\circ t_D.$$
Using the graded map $\Sigma_{D}^{(p)}=t^{(1)}_{D}+\cdots+t^{(p)}_{D}$, then we can modify inductively the graded maps in our contraction (\ref{CoalContr}) and get sequences of graded maps on $\Barc(C)$:
\begin{equation*}
\begin{split}
b_{C}^{(p+1)} & =b_{D}^{(p)}+\Barc(r)\circ t_{D}^{(p)}\circ\Barc(\alpha)=d_C+\Barc(r)\circ\Sigma_{D}^{(p)}\circ\Barc(\alpha)\\
\Barc(\alpha)^{(p+1)} & =\Barc(\alpha)^{(p)}+\Barc(H)\circ t_{D}^{(p)}\circ\Barc(\alpha)=\Barc(\alpha)+\Barc(H)\circ\Sigma_{D}^{(p)}\circ\Barc(\alpha);
\end{split}
\end{equation*}
and sequences of graded maps on $\Barc(D)$:
\begin{equation*}
\begin{split}
\Barc(r)^{(p+1)}& =\Barc(r)^{(p)}+\Barc(r)\circ t_{D}^{(p)}\circ\Barc(H)=\Barc(r)+\Barc(r)\circ\Sigma_{D}^{(p)}\circ\Barc(H)\\
\Barc(H)^{(p+1)} & =\Barc(H)^{(p)}+\Barc(H)\circ t_{D}^{(p)}\circ\Barc(H)=\Barc(H)+\Barc(H)\circ\Sigma_{D}^{(p)}\circ\Barc(H).
\end{split}
\end{equation*}
Now consider the increasing filtration on the bar coalgebra $\Barc(D)$ defined by 
$$F_p\Barc(D)=\begin{cases}{\displaystyle\bigoplus_{n\geqslant 1}^{p}(\sh D)^{\otimes n}} & \textrm{if $p\geqslant 1$}\\
0 & \textrm{otherwise}\end{cases}$$
and the similar filtration on the bar coalgebra $\Barc(C)$. The morphisms $\Barc(\alpha),\Barc(r)$ and $\Barc(H)$ are filtered \emph{i.e.} maps $F_p$ to $F_p$ whereas the morphism $t_D=b_D-d_D$ maps $F_p$ to $F_{p-1}$. As a result we have
$$b_{C}^{(p+1)}=b_{C}^{(p)},\qquad\Barc(\alpha)^{(p+1)}=\Barc(\alpha)^{(p)}$$
on $F_p\Barc(C)$ and 
$$\Barc(r)^{(p+1)}=\Barc(r)^{(p)},\qquad\Barc(H)^{(p+1)}=\Barc(H)^{(p)}$$
on $F_p\Barc(D)$. From this we get a graded map of degree $(1,0)$
$$b_{C}:\Barc(C)\ra\Barc(C)$$
a graded map of degree $(-1,0)$
$$\Barc(H)^{(\infty)}:\Barc(D)\ra\Barc(D)$$
and two graded maps of degree $(0,0)$
$$\xymatrix@C=.6cm{{\Barc(C)}\ar@<.5ex>[r]^{\Barc(\alpha)^{(\infty)}} & {\Barc(D)}\ar@<.5ex>[l]^{\Barc(r)^{(\infty)}}}.$$
The basic perturbation lemma \cite[lemma 3.2]{MR0301736} (see also \cite[lemma 2.1.2]{MR1103672}), shows that the morphisms above provide a contraction
$$\left(\xymatrix@C=.6cm{{(\Barc(C),b_{C})}\ar@<.5ex>[r]^{\Barc(\alpha)^{(\infty)}} & {(\Barc(D),b_D)}\ar@<.5ex>[l]^{\Barc(r)^{(\infty)}}},\Barc(H)^{(\infty)}\right).$$
Moreover $\Barc(\alpha)^{(\infty)},\Barc(r)^{(\infty)}$ are maps of DG coalgebras and $\Barc(H)^{(\infty)}$ is an homotopy of morphisms of DG coalgebras. If we translate this result back to the language of $A_{\infty}$-algebras we exactly get what we wanted. Namely $b_{C}$ provides an $A_{\infty}$-algebra structure on $C$, $\Barc(\alpha)^{(\infty)}$ and $\Barc(r)^{(\infty)} $ provides morphisms of $A_{\infty}$-algebras $\alpha_{\infty}:C\ra D$ and $r_{\infty}:D\ra C$ which lift $\alpha$ and $r$ and satisfy $r_{\infty}\circ\alpha_{\infty}=\bfone$, and finally $\Barc(H)^{(\infty)} $ provides an homotopy of $A_{\infty}$-morphisms between $\bfone$ and $\alpha_{\infty}\circ r_{\infty}$.
\subsection{Higher intersection products on cycle complexes}
Now  apply the perturbation technics as explained in \S \ref{Perturbation} to the contraction 
$$\left(\xymatrix@C=.6cm{{C^*(X,M)}\ar@<.5ex>[r]^{\alpha^{\bf{A}}_X} & {\sC^*(X,M)}\ar@<.5ex>[l]^{r^{\bf{A}}_X}},H^{\bf{A}}_X{''}\right)$$ 
where $\sC^*(X,M)$ is given the $A_{\infty}$-algebra structure constructed in \S \ref{HtpAinf} and $H^{\bf{A}}_X{''}$ is defined from $H^{\bf{A}}_X$ as in \S \ref{Contra}. We get a coderivation $b^\cap_X$ on the bar coalgebra $\Barc C^*(X,M)$ which satisfies $b^{\cap}_{X}\circ b^{\cap}_X$. This is our $A_{\infty}$-algebra structure on $C^*(X,M)$ and we denote by
$$m^{\cap}_{X,n}:C^*(X,M)^{\otimes n}\ra C^*(X,M)$$
the corresponding graded maps of bidegree $(2-n,0)$. We also have two morphisms of DG coalgebras
$$\alpha^{\dagger}_X:\Barc C^*(X,M)\ra\Barc \sC^*(X,M)\qquad r^{\dagger}_X:\Barc\sC^*(X,M)\ra\Barc C^*(X,M) $$
which lift our contraction. We denote by 
$$\alpha^{\infty}_{X,n}:C^*(X,M)^{\otimes n}\ra\sC^*(X,M)\qquad r^{\infty}_{X,n}:\sC^*(X,M)^{\otimes n}\ra C^*(X,M)$$
the morphisms obtained from $\alpha^{\dagger}_{X,n} $ and $r^{\dagger}_{X,n} $ by desuspension as in \S \ref{Desuspension}. 
Let us now show that the higher intersection products still satisfy the projection formula. This is a corollary of proposition \ref{ProjFormulaHomotopy}.
The computation is easy  but nevertheless a bit tedious. We first start by proving the following proposition which will be used to prove the projection formula: 
\begin{prop}\label{ProjFormulaAlpha}
Let $Y\ra X$ be a flat and proper morphism of smooth $k$-schemes of finite type. Let $u\geqslant 1$ be an integer and $r,t$ be nonnegative integers such that $r+1+t=u$. We have
$$\alpha^{\infty}_{X,u}\circ\left(\bfone^{\otimes r}_X\otimes f_*\otimes\bfone^{\otimes t}_X\right)=f_*\circ \alpha^{\infty}_{Y,u}\circ\left((f^*)^{\otimes r}\otimes\bfone_Y\otimes(f^*)^{\otimes t}\right)$$
and also
$$r^{\infty}_{X,u}\circ\left(\bfone^{\otimes r}_X\otimes f_*\otimes\bfone^{\otimes t}_X\right)=f_*\circ r^{\infty}_{Y,u}\circ\left((f^*)^{\otimes r}\otimes\bfone_Y\otimes(f^*)^{\otimes t}\right).$$
\end{prop}
\begin{proof}
To avoid signs and to be able to use the formulas given by the perturbation lemma it is more convenient to prove the two equalities using the bar construction, \emph{i.e.} in terms of $\alpha^{\dagger}_{X,n}$ and $r^{\dagger}_{X,n}$. We have \footnote{In the computations that follow we also denote by $\bfone_X$ the identity of $\sh C^*(X,M)$ or $\sh\sC^*(X,M)$.}
\begin{equation*}
\begin{split}
\left(t_X\circ\Barc(H_X^{\bf{A}}{''})\right)_u &= b^{\bf{A},\cap}_{X,u}\circ\left\lbrack\sum_{i=1}^u\bfone_X^{\otimes i-1}\otimes \sh(H_X^{\bf{A}}{''})\otimes (\sh(\alpha^{\bf{A}}_X)\circ\sh(r^{\bf{A}}_X))^{\otimes u-i}\right]\\
&=\sum_{i=1}^u b^{\bf{A},\cap}_{X,u}\circ\left(\bfone_X^{\otimes i-1}\otimes \sh(H_X^{\bf{A}}{''})\otimes (\sh(\alpha^{\bf{A}}_X)\circ\sh(r^{\bf{A}}_X))^{\otimes u-i}\right).
\end{split}
\end{equation*}
Assume $i=r+1$. Then
\begin{equation*}
\begin{split}
b^{\bf{A},\cap}_{X,u}  \circ\Big(\bfone_X^{\otimes i-1} & \otimes \sh(H_X^{\bf{A}}{''})\otimes  (\sh(\alpha^{\bf{A}}_X)\circ\sh(r^{\bf{A}}_X))^{\otimes u-i}\Big)\circ\left(\bfone^{\otimes r}\otimes \sh(f_*)\otimes\bfone^{\otimes t}\right)\\
& = b^{\bf{A},\cap}_{X,u}\circ\Big(\bfone_X^{\otimes i-1}\otimes \left(\sh(H_X^{\bf{A}}{''})\circ\sh(f_*)\right)\otimes (\sh(\alpha^{\bf{A}}_X)\circ\sh(r^{\bf{A}}_X))^{\otimes u-i}\Big)\\
& = b^{\bf{A},\cap}_{X,u}\circ\Big(\bfone_X^{\otimes i-1}\otimes \left(\sh(f_*)\circ \sh(H_Y^{\bf{A}}{''})\right)\otimes (\sh(\alpha^{\bf{A}}_X)\circ\sh(r^{\bf{A}}_X))^{\otimes u-i}\Big)\\
& =  b^{\bf{A},\cap}_{X,u}\circ\left(\bfone_X^{\otimes r}\otimes\sh(f_*)\otimes\bfone_X^{\otimes t}\right)\circ\Big(\bfone_X^{\otimes i-1}\otimes \sh(H_Y^{\bf{A}}{''})\otimes(\sh(\alpha^{\bf{A}}_X)\circ\sh(r^{\bf{A}}_X))^{\otimes u-i}\Big).
\end{split}
\end{equation*}
Proposition \ref{ProjFormulaHomotopy} yields
\begin{equation*}
\begin{split}
b^{\bf{A},\cap}_{X,u}  \circ\Big(&\bfone_X^{\otimes i-1}\otimes \sh(H_X^{\bf{A}}{''})\otimes  (\sh(\alpha^{\bf{A}}_X)\circ\sh(r^{\bf{A}}_X))^{\otimes u-i}\Big)\circ\left(\bfone_X^{\otimes r}\otimes \sh(f_*)\otimes\bfone_X^{\otimes t}\right)\\
& =\sh(f_*)\circ b^{\bf{A},\cap}_{Y,u}\circ\left(\sh(f^*)^{\otimes r}\otimes\bfone_Y\otimes\sh(f^*)^{\otimes t}\right)\circ\Big(\bfone_Y^{\otimes i-1}\otimes \sh(H_Y^{\bf{A}}{''})\otimes(\sh(\alpha^{\bf{A}}_X)\circ\sh(r^{\bf{A}}_X))^{\otimes u-i}\Big)\\
& = \sh(f_*)\circ b^{\bf{A},\cap}_{Y,u}\circ\Big(\bfone_Y^{\otimes i-1}\otimes \sh(H_Y^{\bf{A}}{''})\otimes(\sh(\alpha^{\bf{A}}_Y)\circ\sh(r^{\bf{A}}_Y))^{\otimes u-i}\Big)\circ\left(\sh(f^*)^{\otimes r}\otimes\bfone_Y\otimes\sh(f^*)^{\otimes t}\right)
 \end{split}
\end{equation*}
since $f^*\circ\alpha^{\bf{A}}_X\circ r^{\bf{A}}_X=\alpha^{\bf{A}}_Y\circ r^{\bf{A}}_Y\circ f^* $. We can do exactly the same kind of computation when $i\neq r+1$ and then taking the sum we get
$$\left(t_X\circ\Barc(H_X^{\bf{A}}{''})\right)_u\circ\left(\bfone_X^{\otimes r}\otimes \sh(f_*)\otimes\bfone_X^{\otimes t}\right)= \sh(f_*)\circ\left(t_Y\circ\Barc(H_Y^{\bf{A}}{''})\right)_u\circ \left(\sh(f^*)^{\otimes r}\otimes\bfone_Y\otimes\sh(f^*)^{\otimes t}\right).$$
Now it is not difficult to check by induction that the projection formula holds for all the maps $t^{(i)}_{X,u}$ where $1\leqslant i\leqslant u$ and consequently holds for $\Sigma^{(u)}_{X,u}$.\par
Now let us check the projection formula for the morphism $\alpha^{\dagger}_{X,u}$. We have $\alpha^{\dagger}_{X,1}=\sh(\alpha^{\bf{A}}_X)$ and $\alpha^{\bf{A}}_X\circ f_*=f_*\circ\alpha^{\bf{A}}_Y$, so we may assume that $u\geqslant 2$. The morphism $\alpha^{\dagger}_X$ is given by
$$\alpha^{\dagger}_X=\Barc(\alpha^{\bf{A}}_X)+\Barc(H^{\bf{A}}_X{''})\circ \Sigma_{X}^{(u)}\circ\Barc(\alpha^{\bf{A}}_X) $$
on $(\sh C^*(X,M))^{\otimes u}$ and therefore $\alpha^{\dagger}_{X,u}$ is given by
$$\alpha^{\dagger}_{X,u}=\sh H^{\bf{A}}_X{''}\circ\Sigma_{X,u}^{(u)}\circ(\sh \alpha^{\bf{A}}_X)^{\otimes u}.$$
The projection formula for $\alpha^{\dagger}_{X,u}$ follows from the projection formula for $\Sigma^{(u)}_{X,u}$ since
\begin{equation*}
\begin{split}
& f_*\circ H^{\bf{A}}_Y{''}=H^{\bf{A}}_X{''}\circ f_*\\
&(\alpha^{\bf{A}}_X)^{\otimes u}\circ\left(\bfone_X^{\otimes r}\otimes f_*\otimes\bfone_X^{\otimes t}\right)=\left(\bfone_X^{\otimes r}\otimes f_*\otimes\bfone_X^{\otimes t}\right)\circ(\alpha^{\bf{A}}_X)^{\otimes u}\\
&(\alpha^{\bf{A}}_Y)^{\otimes u}\circ\left((f^*)^{\otimes r}\otimes\bfone_Y\otimes(f^*)^{\otimes t}\right)=\left((f^*)^{\otimes r}\otimes\bfone_Y\otimes(f^*)^{\otimes t}\right)\circ(\alpha^{\bf{A}}_X)^{\otimes u}.
\end{split}
\end{equation*}
The morphism $r^{\dagger}_X$ is given by
$$r^{\dagger}_X=\Barc(r^{\bf{A}}_X)+\Barc(r^{\bf{A}}_X)\circ \Sigma^{(u)}_{X}\circ\Barc(H^{\bf{A}}_X{''});$$
this yields 
$$r^{\dagger}_{X,u}=\sh r^{\bf{A}}_X\circ\Sigma^{(u)}_{X,u}\circ \left\lbrack\sum_{i=1}^u\bfone^{\otimes i-1}\otimes \sh(H^{\bf{A}}_X{''})\otimes (\sh(\alpha^{\bf{A}}_X)\circ\sh(r^{\bf{A}}_X))^{\otimes u-i}\right\rbrack.$$
The projection formula follows also from the projection formula for the maps $\Sigma^{(u)}_{X,u}$.
\end{proof}
\begin{prop}
Let $Y\ra X$ be a proper and flat morphism of smooth $k$-schemes of finite type. Let $u\geqslant 1$ be an integer and $r,t$ be nonnegative integers such that $r+1+t=u$. We have
$$m^{\cap}_{X,u}\circ\left(\bfone_X^{\otimes r}\otimes f_*\otimes\bfone_X^{\otimes t}\right)=f_*\circ m^{\cap}_{Y,u}\circ\left((f^*)^{\otimes r}\otimes_Y\bfone\otimes(f^*)^{\otimes t}\right).$$
\end{prop}
\begin{proof}
This is equivalent to prove that the morphisms $b^{\cap}_{X,n}$ satisfy the following projection formula:
$$b^{\cap}_{X,u}\circ\left(\bfone_X^{\otimes r}\otimes\sh f_*\otimes\bfone_X^{\otimes t}\right)=\sh f_*\circ b^{\cap}_{Y,u}\circ\left((\sh f^*)^{\otimes r}\otimes\bfone_Y\otimes(\sh f^*)^{\otimes t}\right).$$
Since $\alpha^{\dagger}_X$ is a morphism of $A_{\infty}$-algebras and $r^{\dagger}_X\circ\alpha^{\dagger}_X=\bfone$, we have $b^{\cap}_X=r^{\dagger}_X\circ b^{\bf{A},\cap}_X\circ\alpha^{\dagger}_X$. From this one deduces that
$$b^{\cap}_{X,u}=\sum r^{\dagger}_{X,n+1+m}\circ\left(\bfone_X^{\otimes n}\otimes b^{\bf{A},\cap}_{X,s}\otimes\bfone_X^{\otimes m}\right)\circ\left(\alpha^{\dagger}_{X,i_1}\otimes\cdots\otimes\alpha^{\dagger}_{X,i_k}\right),$$
where the sum runs over all decompositions $i_1+\cdots+i_k=u$ and $n+s+m=k$. Fix such decompositions and let $i_0=0$. There exists an integer $\ell\in\{1,\ldots,k\}$, such that $i_0+\cdots+ i_{\ell-1}\leqslant r+1 \leqslant i_0+\cdots +i_{\ell}$. Set $r_{\ell}=r-i_0-\cdots-i_{\ell-1}$ and let  $t_{\ell}$ be the integer such that $i_{\ell}=r_\ell+1+t_{\ell}$. We have 
$$\left(\alpha^{\dagger}_{X,i_1}\otimes\cdots\otimes\alpha^{\dagger}_{X,i_k}\right)\circ\left(\bfone_X^{\otimes r}\otimes \sh f_*\otimes\bfone_X^{\otimes t}\right)= \alpha^{\dagger}_{X,i_1}\otimes\cdots \otimes\left(\alpha^{\dagger}_{X,i_{\ell}}\circ\left(\bfone^{\otimes r_\ell}_X\otimes\sh f_*\otimes\bfone_X^{\otimes t_{\ell}}\right)\right)\otimes\cdots\otimes\alpha^{\dagger}_{X,i_k} $$
By lemma \ref{ProjFormulaAlpha} 
$$\alpha^{\dagger}_{X,i_{\ell}}\circ\left(\bfone^{\otimes r_\ell}_X\otimes\sh f_*\otimes\bfone_X^{\otimes t_{\ell}}\right)=\sh f_*\circ\alpha^{\dagger}_{Y,i_{\ell}}\circ\left((\sh f^*)^{\otimes r_\ell}\otimes \bfone_Y\otimes(\sh f^*)^{\otimes t_{\ell}}\right);$$
and so we have
\begin{equation*}
\begin{split}
\left(\alpha^{\dagger}_{X,i_1}\otimes\cdots\otimes\alpha^{\dagger}_{X,i_k}\right)\kern-.2em\circ\kern-.2em\left(\bfone_X^{\otimes r}\otimes\sh f_*\otimes\bfone_X^{\otimes t}\right)\kern-.2em=\kern-.2em&\left(\bfone^{\otimes \ell-1}_X\otimes\sh f_*\otimes \bfone_X^{\otimes k-\ell}\right)\kern-.2em\circ\kern-.2em\left(\alpha^{\dagger}_{X,i_1}\otimes\cdots\otimes \alpha^{\dagger}_{Y,i_\ell}\otimes\cdots\otimes\alpha^{\dagger}_{X,i_k}\right)\\
&\circ\left(\bfone_X^{\otimes  i_1}\otimes\cdots\otimes\left((\sh f^*)^{\otimes r_\ell}_X\otimes \bfone_Y\otimes(\sh f^*)^{\otimes t_{\ell}}\right)\otimes\cdots\otimes\bfone_X^{i_k}\right).
\end{split}
\end{equation*}
Assume now that $n\leqslant \ell\leqslant n+s$. In the other cases the computation is similar and even easier. In order to keep notation as simple as possible, we may further assume that $\ell=n+1$ without loss of generality. In that case an application of proposition \ref{ProjFormulaHomotopy} yields
\begin{equation*}
\begin{split}\left(\bfone_X^{\otimes n}\otimes b^{\bf{A},\cap}_{X,s}\otimes\bfone_X^{\otimes m}\right)&\circ\left(\bfone^{\otimes \ell-1}_X\otimes\sh f_*\otimes \bfone_X^{\otimes k-\ell}\right)=\bfone_X^{\otimes n}\otimes\left(b^{\bf{A},\cap}_{X,s}\circ\left(\sh f_*\otimes\bfone_X^{\otimes s-1}\right)\right)\otimes\bfone_X^{\otimes m}\\
&=\bfone_X^{\otimes n}\otimes\left(\sh f_*\circ b^{\bf{A},\cap}_{Y,s}\circ\left(\bfone_Y\otimes(\sh f^*)^{\otimes s-1}\right)\right)\otimes\bfone_X^{\otimes m}\\
&=\left(\bfone_X^{\otimes n}\otimes\sh f_*\otimes\bfone_X^{\otimes m}\right)\kern-.2em\circ\kern-.2em\left(\bfone_X^{\otimes n}\otimes b^{\bf{A},\cap}_{Y,s}\otimes\bfone_X^{\otimes m}\right)\kern-.2em\circ\kern-.2em\left(\bfone_X^{\otimes n}\otimes\bfone_Y\otimes (\sh f^*)^{\otimes s-1}\otimes\bfone_X^{\otimes m}\right).
\end{split}
\end{equation*}
By lemma \ref{ProjFormulaAlpha} 
$$r^{\dagger}_{X,n+1+m}\circ\left(\bfone_X^{\otimes n}\otimes \sh f_*\otimes\bfone_X^{\otimes m}\right)=\sh f_*\circ r^{\dagger}_{Y,n+1+m}\circ\left((\sh f^*)^{\otimes n}\otimes \bfone_Y\otimes(\sh f^*)^{\otimes m}\right).$$
Since $f^*\circ\alpha^{\dagger}_{X,i}=\alpha^{\dagger}_{Y,i}\circ f^*$ for any integer $i$, we obtain 
\begin{equation*}
\begin{split}
r^{\dagger}_{X,n+1+m}&\circ\left(\bfone_X^{\otimes n}\otimes b^{\bf{A},\cap}_{X,s}\otimes\bfone_X^{\otimes m}\right)\circ\left(\alpha^{\dagger}_{X,i_1}\otimes\cdots\otimes\alpha^{\dagger}_{X,i_k}\right)\circ\left(\bfone_X^{\otimes r}\otimes \sh f_*\otimes\bfone_X^{\otimes t}\right)\\
&=\sh f_*\circ r^{\dagger}_{Y,n+1+m}\circ\left(\bfone_Y^{\otimes n}\otimes b^{\bf{A},\cap}_{Y,s}\otimes\bfone_Y^{\otimes m}\right)\kern-.2em\circ\kern-.2em\left(\alpha^{\dagger}_{Y,i_1}\otimes\cdots\otimes\alpha^{\dagger}_{Y,i_k}\right)\kern-.2em\circ\kern-.2em\left((\sh f^*)^{\otimes r}\otimes\bfone_Y\otimes(\sh f^*)^{\otimes t}\right)
\end{split}
\end{equation*}
Taking the sum over all decompositions we finally get 
$$b^{\cap}_{X,u}\circ\left(\bfone_X^{\otimes r}\otimes\sh f_*\otimes\bfone_X^{\otimes t}\right)=\sh f_*\circ b^{\cap}_{Y,u}\circ\left((\sh f^*)^{\otimes r}\otimes_Y\bfone\otimes(\sh f^*)^{\otimes t}\right).$$
This is the projection formula as desired.
\end{proof}

\subsection*{Applications} In a further development we use this $A_{\infty}$-algebra structure to construct an $A_{\infty}$-category of \guillemotleft proper correspondences \guillemotright, whose associated triangulated category of twisted complexes is closely related to the weight complexes defined by H. Gillet and C. Soul\'e.

\appendix
\section{Sign conventions}
The paper \cite{RostChowCoeff} is rather sketchy about sign conventions. When looking carefully at \S 3.9 in \emph{loc.cit.}, one might get puzzled by the somehow strange choices of signs made, moreover in formula 14.3 and the chain rule formula 14.4 of \emph{loc.cit.} the signs are not correct and do not follow from an application of the rules imposed in the definition of a cycle module.  In \emph{loc.cit} the chain complex $C_*(X,M)$ is bigraded, defined by
$$C_p(X,M,n)=\bigoplus_{x\in X_{(p)}}M_{n+p}(\kappa(x));$$
and the definition of the differential for bigraded maps such as the four basic maps only involve the degree of the map with respect with the second grading. This choice is incompatible with the fact that $\mathrm{sign}(f^*)$ should be $1$ for a flat morphism as stated. As one may guess, this is of little importance,  however since we have to work with $A_{\infty}$-algebras we feel the need to be precise on this matter and provide the following tedious definitions.\par
A $\bbZ$-graded complex of abelian groups is a cochain complex $C$ with a decomposition of $C$ into a direct sum over $\bbZ$ of a family of subcomplexes $C(n)$. Let $C$ and $D$ be $\bbZ$-graded complexes of abelian groups.
Graded morphisms $C\ra D$  mapping $C^p(n)$ to $D^{p+r}(n-s)$ are said to be of bidegree $(r,s)$ and form an abelian group $\catgrC^{(r,s)}(C,D)$. The abelian group of graded morphisms of total degree $k$ is then the direct sum
$$\catgrC^k(C,D)=\bigoplus_{r+s=k}\catgrC^{(r,s)}(C,D).$$
These Hom groups are the components a cochain complex with differential given by
$$\delta(u):=d\circ u+(-1)^{k+1}u\circ d $$
for a morphism $u$ of total degree $k$. As usual a graded morphism $u$ is said to be closed when $\delta(u)=0$. The suspension $\sh C$ of  the $\bbZ$-graded complex of abelian groups 
$C$ is defined by $(\sh C)^i(n)=C^{i+1}(n) $ together with the differential given by $d_{\sh C}^i=-d^{i+1}_C$. By definition identities on each components provide a closed morphism $s_C:C\ra\sh C$ of degree $(-1,0)$. The tensor product of the $\bbZ$-graded complexes of abelian groups $C$ and $D$ is defined to be the bigraded abelian group with components
$$(C\otimes D)^k(\ell)=\bigoplus_{\substack{k=p+q\\ \ell=n+m}}C^{p}(n)\otimes_{\bbZ}C^{q}(m) $$
endowed with the differential given by
$$d(a\otimes b)=da\otimes b+(-1)^{p+n}a\otimes db.$$
If $u:C\ra C'$ and $v:D\ra D'$ are graded maps, their tensor product $u\otimes v$ is the graded map given by
$$(u\otimes v)(a\otimes b)=(-1)^{(p+n)(r+s)}u(a)\otimes v(b)$$
where $(r,s)$ is the bidegree of $v$.
With this definition, the differential of the tensor product is simply given by the equality $d_{C\otimes D}=d_C\otimes\bfone+\bfone\otimes d_D$. The commutativity isomorphism is then given as usual by the Koszul rule with respect to the total degree.
With these conventions, $\bbZ$-graded complexes of abelian groups form a tensor DG category \footnote{To define DG-categories we use the sign conventions for complexes of abelian groups induced by the signs convention we have detailed for $\catgrC$.} denoted by $\catgrC$ in this paper.
\backmatter
\bibliographystyle{smfalpha}
\bibliography{HigherChowCoeffbiblio}\end{document}